\documentclass[11pt,a4paper]{article}

\usepackage[headinclude]{typearea}
\usepackage[english]{babel}           
\usepackage[T1]{fontenc}
\usepackage{scrtime}
\usepackage{amsmath,amsthm,amsfonts,amssymb}
\usepackage {color}                                  
\usepackage {pifont}                                 
\usepackage {multicol,multirow}                               
\usepackage[numbers]{natbib}                                                     
\usepackage[normalem]{ulem}                                                     
\usepackage [scaled=.90]{helvet}                     
\usepackage {courier}                                
\usepackage {graphicx}                               
\usepackage {array}                                  
\usepackage {longtable}                              
\usepackage{fancyvrb}
\usepackage{enumerate} 
\usepackage{ifpdf}
\usepackage{caption}
\usepackage{mathrsfs}

\usepackage{setspace}

\usepackage{marvosym}

\usepackage{mathtools}
\mathtoolsset{showonlyrefs} 

\newtheorem{theorem}{Theorem}[section]

\newtheorem{lemma}[theorem]{Lemma}

\theoremstyle{definition}
\newtheorem{assumption}{Assumption}[section]
\newtheorem{definition}[theorem]{Definition}
\newtheorem{example}[theorem]{Example}
\newtheorem{remark}[theorem]{Remark}

\newtheorem{framework}{Framework}[section]

\newcommand{\exclude}[1]{}

\newcommand{\proj}{p}

\newcommand{\1}{{\mathbf{1}}}
\newcommand{\E}{{\mathbb{E}}}

\newcommand{\N}{{\mathbb{N}}}
\renewcommand{\P}{{\mathbb{P}}}

\newcommand{\R}{{\mathbb{R}}}

\newcommand{\F}{\mathbb{F}}

\newcommand{\sign}{\operatorname{sign}}

\definecolor{darkgreen}{rgb}{0,0.5,0}
\definecolor{lightgreen}{rgb}{0.5,0.9,0.5}
\definecolor{magenta}{rgb}{0.75,0,0.25}

\definecolor{violet}{rgb}{0.25,0,0.75}

\newcommand{\hypsurf}{\Theta}
\newcommand{\appx}{X^h}
\newcommand{\appxy}{X^{h}}
\newcommand{\appz}{Z^h}
\newcommand{\eu}{{\underline{u}}}
\newcommand{\es}{{\underline{s}}}
\newcommand{\et}{{\underline{t}}}

\newcommand{\tr}{\operatorname{tr}}

\newcommand{\drift}{A}
\newcommand{\diff}{B}
\newcommand{\smallconsti}{\left(\frac{4}{3 c_0}\right)}

\renewcommand{\P}{{\mathbb P}}

\newcommand{\cF}{{\cal F}}

\newcommand{\be}{\begin{equation}}
\newcommand{\ee}{\end{equation}}
\newcommand{\bea}{\begin{eqnarray}}
\newcommand{\eea}{\end{eqnarray}}
\newcommand{\beast}{\begin{eqnarray*}}
\newcommand{\eeast}{\end{eqnarray*}}
\newcommand{\bproof}{\begin{proof}}
\newcommand{\eproof}{\end{proof}}

\title{An adaptive Euler-Maruyama scheme for stochastic differential equations with discontinuous drift and its convergence analysis}

\author{Andreas Neuenkirch \and Michaela Sz\"olgyenyi \and Lukasz Szpruch}
\date{Preprint, April 2019}

\begin{document}

\maketitle


\begin{abstract}
We study the strong approximation of stochastic differential equations with discontinuous drift coefficients and (possibly) degenerate diffusion coefficients. To account for the discontinuity of the drift coefficient we construct an adaptive step sizing strategy for the explicit Euler-Maruyama scheme. As a result, we obtain a numerical method which has -- up to logarithmic terms -- strong convergence order $1/2$ with respect to the average computational cost.  We support our theoretical findings with several numerical examples.\\

\end{abstract}


\section{Introduction and Main Results}\label{sec:intro}

In this manuscript, we consider the strong approximation of  time-homogeneous It\=o-stochastic differential equations (SDEs) of the form
\begin{align}\label{eq:sde}
dX_t=\mu(X_t) dt + \sigma(X_t) dW_t, \quad t \geq 0, \qquad X_0=x,
\end{align} where $x\in \R^d$ is the initial value,  $\mu\colon\R^d \rightarrow \R^d$ is the drift coefficient, $\sigma\colon\R^d \rightarrow \R^{d , d}$ is the diffusion coefficient and $W=(W_{t})_{t \geq 0}$ is a $d$-dimensional Brownian motion.
In contrast to most of the analysis in the literature, we allow 
\begin{itemize}\setlength{\itemsep}{0em}
 \item[(i)] the drift coefficient $\mu$ to be discontinuous on a hypersurface $\hypsurf$,
 \item[(ii)] and the diffusion coefficient $\sigma$ to be degenerate outside $\hypsurf$.
\end{itemize}

Our aim is to construct a numerical scheme, which is relatively easy to implement and has root mean square convergence order $1/2$ in terms of the computational cost for a large class of SDEs. 
So far only the transformation-based Euler-Maruyama scheme given in \cite{sz2016b} for SDE \eqref{eq:sde} is known to have this property for multi-dimensional SDEs.
\\

To state our main results denote the distance to the exceptional set $\hypsurf$ by $$d(x,\hypsurf)=\inf\{\|x-y\|:y \in \hypsurf\}, \quad x\in\R^d,$$ and for every $\varepsilon>0$ define
$$\hypsurf^\varepsilon:=\{x\in \R^d: d(x,\hypsurf)<\varepsilon\}.$$ We  consider the  adaptive Euler-Maruyama scheme given by
\begin{align} \label{euler_anf}
 \appx_0 &=x,\qquad  \tau_0 = 0,
 \end{align}
 and
 \begin{align} \label{euler_it}
\appx_{\tau_{k+1}} = \appx_{\tau_k}+ \mu(\appx_{\tau_k})(\tau_{k+1}-\tau_k) + \sigma(\appx_{\tau_k})(W_{\tau_{k+1}}-W_{\tau_k}), \qquad
 \tau_{k+1}&= \tau_{k} + h(\appx_{\tau_k},\delta),
\end{align}
with $ k\in \N_0$, and step size function $h\colon\R^d\times (0,1)\to(0,1)$,

\begin{align}\label{eq:step size}
  h(x,\delta)=\begin{cases}
    \delta^2,\, & x \in \hypsurf^{\varepsilon_2},\\
   \frac{1}{\sup_{x \in \hypsurf^{\varepsilon_0}} \| \sigma(x) \|^2} \left( \frac{d(x,\hypsurf)}{  \log (1/\delta)}\right)^2, \, & x \in \hypsurf^{\varepsilon_1} \backslash  \hypsurf^{\varepsilon_2},\\
   \delta,\, & x \notin  \hypsurf^{\varepsilon_1},
  \end{cases}
 \end{align} 
where \begin{align} \label{eq:step_bd}\varepsilon_1=\sup_{x \in \hypsurf^{\varepsilon_0}} \| \sigma(x) \|  \log(1/\delta)\sqrt{\delta}, \qquad \varepsilon_2=\sup_{x \in \hypsurf^{\varepsilon_0}} \| \sigma(x) \| \log(1/\delta)\delta, \end{align}
with $\delta \in (0,1)$ and $\varepsilon_0 > 4 \varepsilon_1 > 4 \varepsilon_2$. Note that $\tau$ depends on $h(\cdot, \delta)$ and also $\varepsilon_1, \varepsilon_2$ depend on $\delta$, but to simplify the notation we suppress this dependence.
For  mathematical convenience we will work with the continuous time Euler-Maruyama scheme, i.e.~between discretization points we set
\begin{align} \appx_t &= \appx_{\tau_k}+ \mu(\appx_{\tau_k})(t-\tau_k) + \sigma(\appx_{\tau_k})(W_{t}-W_{\tau_k}) , \quad  t \in [\tau_k, \tau_{k+1}]. \label{eq:euler-timecont} \end{align}

Obviously this scheme uses smaller steps close to the discontinuities, has maximal step size $\delta$, minimal step size $\delta^2$, and interpolates both step sizes in an intermediate regime. The step sizing strategy arises from optimally balancing Gaussian tail estimates and occupation time estimates of the Euler-Maruyama scheme, which in particular accounts for the log-terms.\\

The computational cost of $\appx$ on $[0,T]$, i.e.~the number of arithmetic operations, function evaluations, in particular of $\mu$, $\sigma$ and $h$, and random numbers, is  proportional to the number of steps which are needed to reach time $T$, that is
\begin{align} N(h,\delta)= \min \{k\in \mathbb{N}: \,  \tau_k \geq T \}. \end{align}
We will use this quantity as a proxy for the computational cost of the scheme. Clearly, the evaluation of $h$ might be a non-trivial problem, which is however out of the scope of the present work.

We will work under mild assumptions, i.e. 
\begin{itemize}\setlength{\itemsep}{0em}
 \item[(i)] $\mu$ is supposed to be piecewise Lipschitz and its discontinuity set $\hypsurf$ is a sufficiently regular hypersurface,
 \item[(ii)] $\sigma$ is globally Lipschitz,
 \item[(iii)] $\mu$ and $\sigma$  satisfy a geometric smoothness and boundedness condition close to $\hypsurf$,
\end{itemize}
see  Assumption \ref{ass:existence}.\\

For fixed $T>0$ we will show that
\begin{align}\left( \mathbb{E} \left[\sup_{s \in [0,T]} \|X_s-\appx_s\|^2\right] \right)^{1/2} \leq C_{\textrm{rmse}} \cdot \sqrt{1+\log(1/\delta)} \sqrt{\delta}, \end{align}
see Theorem \ref{thm:conv}, and
\begin{align} \E [N(h,\delta)] \leq C_{\textrm{cost}} \cdot (1+\log(1/\delta))\delta^{-1}, \end{align}  see Theorem \ref{thm:cost},
for some constants $C_{\textrm{rmse}}, C_{\textrm{cost}} >0$ depending only on $\mu$, $\sigma$, $\hypsurf$, $T$, $x$.

So up to logarithmic terms the adaptive Euler-Maruyama scheme recovers the classical order $1/2$ with respect to the average computational cost.
Note that both theorems remain valid if
the quantity $\sup_{x \in \hypsurf^{\varepsilon_0}} \| \sigma(x) \|$ is  replaced by an upper bound  in the definition of $\varepsilon_1, \varepsilon_2$ and $h$.

The remainder of this article is structured as follows: in the following subsections we briefly review recent results on the approximation of SDEs with discontinuous coefficients and adaptive numerical methods for SDEs. Sections \ref{sec:preliminaries} and \ref{prop:Euler} contain preliminary and auxiliary results, while Sections \ref{sec:convergence} and \ref{sec:cost} contain the error and cost analysis.
Section \ref{sec:examples} provides some numerical examples.

\subsection{Numerical methods for SDEs with discontinuous coefficients}\label{history}

Typically, existence and uniqueness results for SDEs only allow discontinuities in the drift coefficient, but not in the diffusion coefficient, see, e.g.,
\cite{
veretennikov1984} and the recent works \cite{sz14,sz2016a}. Thus -- unless otherwise mentioned -- the diffusion coefficient is globally Lipschitz for the following methods and results.

Up to the best of our knowledge the first contribution is
 \cite{gyongy1998}. In this work, the almost sure convergence for an SDE with possibly discontinuous drift coefficient is established, as long the drift coefficient is still one-sided Lipschitz, the diffusion coefficient is locally Lipschitz and there exists
 a Lyapunov function for the SDE. For SDEs with additive noise the results of \cite{halidias2008}  provide   strong  convergence of the Euler-Maruyama scheme for  discontinuous, but monotone drift coefficients.

Recently several contributions for strong approximation have been given in  a series of articles of Ngo and Taguchi   \cite{ngo2016a,ngo2016b, ngo2016c} and   Leobacher and  Sz\"olgyenyi  \cite{sz15,sz2016b,sz2017c}.
For the multi-dimensional SDE \eqref{eq:sde} these works provide 
\begin{itemize}\setlength{\itemsep}{0em}
 \item[(i)] the $L^1$-convergence order $1/2$  for the equidistant Euler-Maruyama scheme, if $\mu$ is one-sided Lipschitz and an appropriate limit of smooth functions, and $\sigma$ is bounded and uniformly non-degenerate, see \cite{ngo2016c},
 \item[(ii)] the $L^2$-convergence order $1/4-\epsilon$ for arbitrarily small $\epsilon>0$ for the equidistant Euler-Maruyama scheme under Assumption  \ref{ass:existence}  and additionally the boundedness of $\mu$ and $\sigma$, see \cite{sz2017c},
 \item[(iii)] the $L^2$-convergence order $1/2$ for a transformation based Euler-Maruyama method under Assumption  \ref{ass:existence}, see \cite{sz2016b}. However, this transformation is in general difficult to compute, which limits its applicability.
 \end{itemize}

After our work was prepared and submitted, M\"uller-Gronbach and Yaroslavtseva obtained for scalar SDEs  and the equidistant Euler scheme the $L^p$-convergence order $1/2$, for any $p \geq 1$,  see \cite{TMG-Y}. Whether such an improvement of (ii) is possible also in the multi-dimensional case,  presently remains an open question.

The weak approximation of SDEs with discontinuous coefficients has been studied in \cite{kohatsu2013}, where an Euler-type scheme  based on an SDE with mollified drift coefficient is analyzed, and in \cite{frikha}, where an error bound for the density of the Euler-Maruyama scheme for (skew) diffusions is obtained.
Finally, for scalar SDEs with additive noise, \cite{etore2014} provides a simulation scheme based on an approximation by skew perturbed SDEs.
\medskip

\subsection{Adaptive step sizing procedures for strong approximation of SDEs}

Adaptive timestepping strategies have turned out to be very effective in probabilistic numerical analysis, though their error  analysis typically provides mathematical challenges.

\begin{itemize}\setlength{\itemsep}{0em}
 \item 
The early works on adaptive methods propose strategies based on local error estimators in analogy to numerical methods for ordinary differential equations, see, e.g.,
 \cite{MR1470933}, \cite{burx2}, and \cite{MR2017023}.
 \item 
Adaptive methods have also been used to preserve  ergodic properties of the underlying SDE, see, e.g., \cite{MR2337577,MR1931266}. 
In fact,  to recover ergodicity by adaptivity has already been proposed in \cite{MR1440273}. 
\item 

For optimal approximation of SDEs in the Information-Based-Complexity framework adaptive methods have been exhaustively analyzed in a series of articles  \cite{MR2099646,MR1910644,MR1817611,MR1677407}.
In these works,  optimal convergence rates and asymptotically optimal schemes  have been established for various error criteria.
\item 
Finally, step adaptation strategies for SDEs with non-globally Lipschitz coefficients have been studied in \cite{kelly,FangGiles}. This research has been partially motivated for the purpose of multilevel Monte Carlo simulations \cite{lester2015adaptive,hoel2012adaptive}.

\end{itemize}

\section{Preliminaries}
\label{sec:preliminaries}

In this section we present some notions from differential geometry and analysis, state the assumptions we make on the coefficients of  SDE \eqref{eq:sde} together with an existence and uniqueness result, and present a Krylov-type estimate for It\=o processes.

All stochastic variables introduced in the following are assumed to be defined on the filtered probability space
$(\Omega,\cF,\F,\P)$, where $\F=(\F_t)_{t\ge 0}$ is a normal filtration. In particular, $W$ is a $d$-dimensional $(\Omega,\cF,\F,\P)$-Brownian motion.

\subsection{Definitions from differential geometry and analysis}

In order to allow for discontinuities of the drift, we replace the usual global Lipschitz condition by the \emph{piecewise Lipschitz condition}, which was first introduced in \cite{sz2016b}.
For this, we recall two definitions.

\begin{definition}[\text{\cite[Definitions 3.1 and 3.2]{sz2016b}}]
Let $A\subseteq \R^d$.
\begin{enumerate}
\item For a continuous curve $\gamma\colon [0,1]\to \R^d$, let  $\ell(\gamma)$ denote its length, i.e.
\[
\ell(\gamma)=\sup_{n \in \mathbb{N}, \, 0\le t_1<\ldots<t_n\le 1}\sum_{k=1}^n \|\gamma(t_k)-\gamma(t_{k-1})\|.
\] 
The {\em intrinsic metric} $\rho$ on $A$ 
is given by 
\[
\rho(x,y):=\inf\{\ell(\gamma)\colon \gamma\colon[0,1]\to A \text{ is a continuous curve satisfying } \gamma(0)=x,\, \gamma(1)=y\},
\]
where $\rho(x,y):=\infty$, if there is no continuous curve from $x$
to $y$.  
\item Let $f\colon A\to \R^m$ be a function.
We say that $f$ is {\em intrinsic Lipschitz}, if it is Lipschitz w.r.t. the
intrinsic metric on $A$, i.e.~if there exists a constant $L>0$ such that
\[
\forall x,y\in A\colon \|f(x)-f(y)\|\le L \rho(x,y).
\]
\end{enumerate}
\end{definition}

Of course every Lipschitz function is intrinsic Lipschitz, but the reverse does not hold. 

\begin{definition}[\text{\cite[Definition 3.4]{sz2016b}}]\label{def:pw-lip}
A function $f\colon \R^d\to\R^m$ is {\em piecewise Lipschitz}, if
there exists a hypersurface $\hypsurf$ with finitely many connected 
components and with the property, that
the restriction $f|_{\R^d\backslash \hypsurf}$ is intrinsic Lipschitz.
We call $\hypsurf$ an {\em exceptional set} for $f$,
and we call
\[
\sup_{x,y\in \R^d\backslash\hypsurf}\frac{\|f(x)-f(y)\|}{\rho(x,y)}
\]
the \emph{piecewise Lipschitz constant} of $f$.
\end{definition}

The following example shows, why it is necessary to resort to the intrinsic metric in the definition of the piecewise Lipschitz condition.
\begin{example}\label{ex:pw-lip-intrinsic}
Consider a function $f\colon \R^2\to \R^2$ which is piecewise Lipschitz and discontinuous at $\hypsurf=\{(x_1,x_2)\in \R^2\colon x_1<0\}$.
To obtain a Lipschitz estimate of, e.g., $\|f(-x_1,x_1)-f(-x_1,-x_1)\|$ for $x_1\in\R$, the Euclidean metric cannot be used, since the direct connection of $(-x_1,x_1)$ and $(-x_1,-x_1)$ crosses $\hypsurf$. The intrinsic metric provides a Lipschitz estimate with a connecting curve that lies in $\R^2\backslash\hypsurf$.
\end{example}

In the following, we consider piecewise Lipschitz functions with exceptional set $\hypsurf$, where $\hypsurf$ is a fixed, sufficiently regular hypersurface, see Assumption \ref{ass:existence}.\ref{ass:existence-pwlip} below.
We denote the Lipschitz constant of a function $f$ if it is finite, and otherwise its piecewise Lipschitz constant, by $L_f$.
For a function $f\colon\R^d \rightarrow \R$ we denote
$\|f\|_{\infty,\hypsurf^{\varepsilon_0}}:= \sup_{x \in \hypsurf^{\varepsilon_0}} \| f(x)\|$.

If $\hypsurf\in C^4$, locally there exists a unit normal vector, that is a continuously differentiable $C^3$
function $n\colon U\subseteq\Theta\to \R^d$ such that for every $\zeta\in U$, $\|n(\zeta)\|=1$, and $n(\zeta)$ is orthogonal to the tangent space of $\hypsurf$ in $\zeta$.

Recall the following definition from differential geometry:

\begin{definition}\label{def:positivereach}
Let $\hypsurf \subseteq \R^d$. 
\begin{enumerate}
\item \label{def:ucpp} 
An environment $\hypsurf^\varepsilon$ is said to have the 
{\em unique closest point property}, if for every $x\in \R^d$
with $d(x,\hypsurf)<\varepsilon$ there is a unique $\proj\in \hypsurf$ with
$d(x,\hypsurf)=\|x-\proj\|$.
Therefore, we can define a mapping 
$\proj\colon\hypsurf^{\varepsilon}\to \hypsurf$ assigning to each $x$
the point $\proj(x)$ in $\hypsurf$, which is closest to $x$.
\item 
A set $\hypsurf$ is said to be of {\em positive reach}, if there exists  
$\varepsilon>0$ such that  $\hypsurf^\varepsilon$ has the  
unique closest point property.
The {\em reach} $r_{\hypsurf}$ of $\hypsurf$ is the supremum 
over all such $\varepsilon$ if such an $\varepsilon$ exists, and 
0 otherwise. 
\end{enumerate}

\end{definition}

\subsection{Existence of a unique strong solution}
\label{def_g_etal}

The main results in this paper require the following set of assumptions:
\begin{assumption}\label{ass:existence}
We assume for the coefficients $\mu \colon \mathbb{R}^d \rightarrow \mathbb{R}^d$ and $\sigma\colon \mathbb{R}^d \rightarrow \mathbb{R}^{d , d}$ of SDE \eqref{eq:sde}:
\begin{enumerate}\setlength{\itemsep}{0em}
\item \label{ass:existence-sigmalip} the diffusion coefficient $\sigma$  is Lipschitz;
\item \label{ass:existence-pwlip} the drift coefficient $\mu$ is a piecewise Lipschitz function;
its exceptional set $\hypsurf$ is a $C^4$-hypersurface with reach $r_{\hypsurf}> \varepsilon_0$ for some $\varepsilon_0>0$ and every unit normal vector $n$ of $\hypsurf$ has bounded second and third derivative;
\item\label{ass:existence-musigmalocbounded} the coefficients $\mu$ and $\sigma$ satisfy
$$  \sup_{x \in \hypsurf^{\varepsilon_0}} (\|\mu(x)\| + \|\sigma(x)\|) < \infty;     $$
\item \label{ass:existence-nonparallelity} {\em (non-parallelity condition)}\,
there exists a constant $c_0>0$ such that $\|\sigma(\xi)^\top n(\xi)\|\ge c_0$ for
all $\xi\in \hypsurf$;
\item \label{ass:existence-alpha} the function
\begin{align}\label{eq:alphad}
\alpha\colon \hypsurf \rightarrow \R^d, \qquad \alpha(\xi)=\lim_{h\to 0+}\frac{\mu(\xi-h n(\xi))-\mu(\xi+hn(\xi))}{2 \|\sigma(\xi)^\top n(\xi)\|^2}
\end{align} is well defined, bounded, and
belongs to $C^3_b(\hypsurf;\mathbb{R}^d)$.
\end{enumerate}
\end{assumption}
Note that $\mu$ and $\sigma$ satisfy a linear growth condition due to Assumptions \ref{ass:existence}.\ref{ass:existence-sigmalip}, \ref{ass:existence}.\ref{ass:existence-pwlip}, and \ref{ass:existence}.\ref{ass:existence-musigmalocbounded}.

\paragraph{Remark on Assumption \ref{ass:existence}:}
\hfill
\begin{enumerate}\setlength{\itemsep}{0em}
 \item Assumption \ref{ass:existence}.\ref{ass:existence-pwlip} is needed to be able to locally flatten $\hypsurf$ to a plane in a regular way.
 Furthermore, it guarantees that $n'$ is bounded, see \cite[Lemma 3.10]{sz2016b}.
 \item Assumption \ref{ass:existence}.\ref{ass:existence-nonparallelity} ensures that $\sigma(\xi)$ has a component orthogonal to $\hypsurf$ for all $\xi \in \hypsurf$. It is significantly weaker than the uniform ellipticity condition which is usually required in the literature on SDEs with discontinuous drift;
 \item Assumption \ref{ass:existence}.\ref{ass:existence-alpha} is a technical
condition that is required for the transformation method from  \cite{sz2016b}, which is the basis of our convergence proof, to work.
\item Assumption \ref{ass:existence} is satisfied if, e.g.,
\begin{enumerate}[(i)]
 \item the exceptional set 
$\hypsurf$ is a  compact set. Note that then its complement satisfies
$\R^d\backslash\hypsurf=A_1\cup\dots\cup A_n$
where $A_1, \ldots, A_n$ are open and connected subsets of $\R^d$,
\item
there exist  Lipschitz $C^3$-functions 
$\mu_1,\dots,\mu_n\colon\R^d\longrightarrow\R^d$
such that $\mu=\sum_{k=1}^n \1_{A_k}\mu_k$ and $\sigma$ is Lipschitz and $C^3$,
\item and Assumptions \ref{ass:existence}.\ref{ass:existence-musigmalocbounded} and  \ref{ass:existence}.\ref{ass:existence-nonparallelity} hold.
\end{enumerate}
Compare Example 2.6 in \cite{sz2017c}.
\end{enumerate}

\begin{theorem}[\text{\cite[Theorem 3.21]{sz2016b}}]
Let Assumption \ref{ass:existence} hold. Then SDE \eqref{eq:sde} has a unique strong solution.
\end{theorem}

The proof of the above theorem and also the proof of our convergence result rely on a mapping $G\colon \mathbb{R}^d \rightarrow \mathbb{R}^d$, which transforms the SDE for $X$ in another SDE which has Lipschitz coefficients.
More precisely, we define  
\begin{align}\label{eq:G}
G(x)=\begin{cases}
 x+\varphi(x) \alpha(\proj(x)),&  x\in \hypsurf^{\varepsilon_0},\\
x, & x\in \R^d\backslash \hypsurf^{\varepsilon_0},
\end{cases}
\end{align}
with $r_{\hypsurf}>\varepsilon_0>0$, see 
Assumption \ref{ass:existence}.\ref{ass:existence-pwlip}, 
$\alpha$ as in Assumption \ref{ass:existence}.\ref{ass:existence-alpha}, and 
\begin{align}  \label{eq:varphi}
 \varphi(x)=n(\proj(x))^\top(x-\proj(x))
\|x-\proj(x)\|\phi\left(\frac{\|x-\proj(x)\|}{c}\right),                                                                                                                                       
 \end{align}
with a constant $c>0$ and $\phi\colon \R \to \R$,
\begin{align*}
\phi(u)=
\begin{cases}
(1+u)^4 (1-u)^4, & |u|\le 1,\\
0, & |u|> 1.
\end{cases}
\end{align*}
The map $G$ has the following properties: \newpage
\begin{lemma}\label{lem:propG}
Let Assumption \ref{ass:existence} hold. Then we have
\begin{enumerate}[(i)]\setlength{\itemsep}{0em}
 \item\label{it:C1} $G\in C^1(\mathbb{R}^d,\mathbb{R}^d)$;
 \item\label{it:GpwLip} $G'$ is Lipschitz, $G''$ exists on $\mathbb{R}^d \setminus \hypsurf$ and is piecewise Lipschitz with exceptional set $\hypsurf$;
 \item\label{it:boundedder} $G'$ and $G''$ are bounded;
 \item\label{it:Ginv} for $c$ sufficiently small, see \cite[Lemma 1]{aap-corr}, $G$ is globally invertible;
 \item\label{it:GLip} $G$ and $G^{-1}$ are Lipschitz continuous;
 \item\label{it:Gito} It\=o's formula holds for $G$ and $G^{-1}$.
\end{enumerate}
\end{lemma}
\begin{proof} For (i) and (iv) see \cite[Theorem 3.14]{sz2016b}. Assertion (v) follows from the proof
of \cite[Theorem 3.20]{sz2016b} and  for (vi) see  \cite[Theorem 3.19]{sz2016b}.
The boundedness of $G'$ follows from (v),  and the boundedness of $G''$ is proven in \cite[Lemma 4]{aap-corr}, showing (iii).
Moreover this, \cite[Lemma 3.6]{sz2016b}, \cite[Lemma 3.8]{sz2016b}, and \cite[Lemma 3.11]{sz2016b} assure that $G'$ is Lipschitz.
Hence, assertion (ii) follows again from the proof of  \cite[Theorem 3.20]{sz2016b} and the fact $G'''$ is bounded on $\mathbb{R}^d \setminus \hypsurf$, see \cite[Lemma 4]{aap-corr}.
\end{proof}

Now, define the coefficients
\begin{equation} \label{musig_G} 
\begin{aligned} 
\mu_G(z)&=G'(G^{-1}(z))\mu(G^{-1}(z))+\frac{1}{2}\tr \left[\sigma(G^{-1}(z))^\top G''(G^{-1}(z))\sigma(G^{-1}(z))\right],\\
\sigma_G(z)&=G'(G^{-1}(z)) \sigma(G^{-1}(z)) ,
\end{aligned}
\end{equation}
for $z \in \R^d$.
\begin{lemma}[\mbox{\cite[Theorem 3.20]{sz2016b}}]\label{lem:propG_lip}
The functions $\mu_G\colon \R^d \rightarrow \R^d$ and $\sigma_G\colon\R^d \rightarrow \R^{d,d}$ are globally Lipschitz.
\end{lemma}
Thus, the SDE 
\begin{align}\label{eq:SDEtransf}
 dZ_t = \mu_G(Z_t) dt+\sigma_G(Z_t) dW_t, \quad t\ge0, \qquad Z_0=G(x),
\end{align}
has a unique strong solution $Z=(Z_{t})_{t\ge0}$. Moreover, $X=(X_t)_{t\ge 0}$ given by $X_t=G^{-1}(Z_t)$ solves SDE \eqref{eq:sde}, which follows from an application of It\=o's formula and \cite[Theorem 3.19]{sz2016b}.

\subsection{Occupation time estimates for It\=o processes}
\label{subsec:occhypsurf}
In this  subsection we study the occupation time of an It\=o process
close to a $C^4$-hypersurface.  The following result is a slight extension of  \cite[Theorem 2.7]{sz2017c} and is sometimes refereed to as Krylov's estimate. 
While classically Krylov estimates are derived for non-degenerate diffusions, see \cite{krylov1980}, 
the following result only assumes non-degeneracy of the diffusion coefficient at a normal direction within an environment of the hypersurface $\hypsurf$. 

\begin{theorem}\label{th:occtime} 
Let $\hypsurf$ be a $C^4$-hypersurface of positive reach  and let $r_{\hypsurf}>{\epsilon_0}>0$. Let further
$A=(A_t)_{t \geq 0}$, $B=(B_t)_{t \geq 0}$ be $\R^d$, respectively $\R^{d,d}$-valued progressively measurable processes such that
$$ \int_0^t \mathbb{E} \left[ \|A_s\| + \| B_s \|^2 \right] ds < \infty, \quad t\ge 0.$$
Moreover, let 
$X=(X_t)_{t\ge 0}$ be the $\R^d$-valued It\=o process given by
\[
X_t=X_0+\int_0^t \drift_s ds+\int_0^t \diff_s dW_s, \quad t \geq 0,
\] with $X_0 \in \R^d$.
Assume finally that
\begin{enumerate} 
\item[(i)] \label{it:bounded-coeff}
there exists a constant $c_{AB}>0$ such that 
for almost all $\omega\in \Omega$ we have
\[
\forall t\in[0,T]: \, X_t(\omega)\in \hypsurf^{\epsilon_0} \Longrightarrow
\max(\|\drift_t(\omega)\|,\|\diff_t(\omega)\|)\le c_{AB};
\] 
\item[(ii)]
there exists a constant $c_0>0$ such that 
for almost all $\omega\in \Omega$  we have
\begin{equation*} 
\forall t\in[0,T]: X_t(\omega)\in \hypsurf^{\epsilon_0} \Longrightarrow
n (p(X_t(\omega)) )^\top\diff_t(\omega)\diff_t(\omega)^\top n (p(X_t(\omega) ))\ge c_0^2.
\end{equation*} 
\end{enumerate}
Then there exists a constant $C>0$ such that for all 
$0<\epsilon<\epsilon_0/2$ and any measurable function $f:[0,\infty) \rightarrow [0,\infty)$ we have
\[
\E \left[ \int_0^{T}  f(d(X_s,\hypsurf)) \1_{\{X_s \in \hypsurf^\epsilon \}} ds \right] 
\le C \int_0^{\epsilon} f(x) dx .
\]
\end{theorem}

\begin{proof}
The proof relies on \cite[Theorem 2.7]{sz2017c}, where a scalar process $Y$ is constructed  such that the occupation time of $Y$ in an
environment of $\{0\}$ is the same as the occupation time of $X$ in an environment of $\hypsurf$.
More precisely, there exists a bounded real-valued  It\=o process 
$$Y_t=Y_0 + \int_0^t \hat \drift_s ds+  \int_0^t \hat \diff_s dW_s, \quad t \geq 0,$$ 
where  $\hat \drift$, $\hat \diff$ are uniformly bounded, progressively measurable, $\mathbb{R}$ respectively $\mathbb{R}^{1,d}$-valued processes, such that
 \begin{align} 
 \label{Y-1} Y_t  \cdot \1_{ \{X_t \in \hypsurf^{\epsilon_1} \} }= \lambda  ( D(X_t)) \cdot \1_{ \{X_t \in \hypsurf^{\epsilon_1} \} }, \quad t \geq 0,
  \end{align}
where $D(x)=n(p(x))^{\top}(x -p(x))$ for $x\in\Theta^{\epsilon_0}$, $\epsilon_1=\epsilon_0/2$, and $\lambda: \R \to \R$ is given by 
\[
\lambda(z)= \begin{cases}
z- \frac{2}{3  \epsilon_1^2}z^3 +\frac{1}{5 \epsilon_1^4}z^5,
& |z|\le  \epsilon_1,\\
 \frac{8  \epsilon_1}{15}, & z>  \epsilon_1,\\
-\frac{8 \epsilon_1}{15}, & z< - \epsilon_1.
\end{cases}
\]
Since $|D(X_t)|=d(X_t, \hypsurf)$, the value of $Y$ corresponds to the $\lambda$-transformed signed distance of $X$ to $\hypsurf$.
Since 
$\lambda'(\pm\epsilon_1)=\lambda''(\pm \epsilon_1)=0$,
it holds that $\lambda\in C^2$. Moreover, $\lambda\colon[- \epsilon_1,  \epsilon_1] \rightarrow [\lambda(-\epsilon_1), \lambda( \epsilon_1)]$ is invertible.

By construction the quadratic variation of $Y$ satisfies $\P$-a.s.~that
$$ \int_0^t \1_{\left\{Y_s\in (-\lambda(\epsilon),\lambda(\epsilon))\right\}} d[Y]_s= \int_0^t \1_{\left\{Y_s\in (-\lambda(\epsilon),\lambda(\epsilon))\right\}}
\left|\lambda'\left(D(X_s)\right) \right|^2
n(p(X_s))^\top\diff_s\diff_s^\top n(p(X_s))ds, \,\, t \geq 0,$$
see the proof of \cite[Theorem 2.7]{sz2017c}.
Since 
\begin{align} 
 \label{Y-2} |\lambda^{-1}( Y_t)|  \cdot \1_{ \{X_t \in \hypsurf^{\epsilon_1} \} }=  \big | D(X_t) \big|\cdot \1_{ \{X_t \in \hypsurf^{\epsilon_1} \} } = d(X_t, \hypsurf) \cdot \1_{ \{X_t \in \hypsurf^{\epsilon_1} \} }, \quad t \geq 0, \end{align}
and
$\lambda'(z)\ge \left(\frac{3}{4}\right)^2$ for all $|z|\le \epsilon\le \epsilon_0/2 $, Assumption (ii) assures that
$$ \E \left[ \int_0^T  f(d(X_s,\theta)) \1_{\{X_s\in \hypsurf^\epsilon\}} ds \right]
\le \smallconsti^{2}
\E\left[\int_0^T f(|\lambda^{-1}(Y_s)|)\1_{\left\{Y_s\in (-\lambda(\epsilon),\lambda(\epsilon))\right\}} d\left[Y\right]_s\right].$$

Therefore, the occupation  time formula 
\cite[Chapter 3, 7.1 Theorem]{karatzas1991} for one-dimensional continuous 
semimartingales yields
\begin{align*}
\E \left[ \int_0^T f(d(X_s,\theta)) \1_{\{X_s\in \hypsurf^\epsilon\}} ds \right]
&\le \smallconsti^{2}
\E\left[\int_0^T f(|\lambda^{-1}(Y_s)|)\1_{\left\{Y_s\in (-\lambda(\epsilon),\lambda(\epsilon))\right\}} d\left[Y\right]_s\right]\\
&= \smallconsti^{2} \E\left[\int_\R f(|\lambda^{-1}(y)|) \1_{(-\lambda(\epsilon),\lambda(\epsilon))}(y)L_T^{y}\left(Y\right)dy\right].
\end{align*}
By the Tanaka-Meyer formula, see e.g 
\cite[Chapter 3, 7.1 Theorem]{karatzas1991}, we have
$$ 2 \E  \left[  L_T^{y} \right]=\E  \left[  |Y_T-y| -    |Y_0-y| \right]- \int_0^T \E  \left[  \sign(Y_s-y)\hat{B}_s \right] ds $$
and therefore 
$$  \sup_{y\in \R}\E  \left[   L_T^{y}  \right] \leq \E \left[  |Y_T-Y_0|  \right]+  \int_0^T \E  \left[  | \hat{B}_s| \right]ds \leq   \E  \left[  |Y_T-Y_0| \right] + T \|\hat{B}\|_{\infty} < \infty,$$
since $Y$ is a real-valued It\=o process with bounded coefficients.
Thus, it follows that
\begin{align*}
\E \left[ \int_0^T  f(d(X_s,\theta)) \1_{\{X_s\in \hypsurf^\epsilon\}} ds \right]
&\le  \frac{2^5}{3^2c_0^2}\, \sup_{y\in \R}\E\left[L_T^{y}\left(Y\right)\right] \int_0^{\lambda(\epsilon)} f( \lambda^{-1}(x)) dx.
\end{align*}
Since $0 \leq \lambda'(x) \leq 1$ for all $x \in \mathbb{R}$, substitution yields
\begin{align*}
\E \left[ \int_0^T f(d(X_s,\theta)) \1_{\{X_s\in \hypsurf^\epsilon\}} ds \right]
&\le  \frac{2^5}{3^2c_0^2}\, \sup_{y\in \R}\E\left[L_T^{y}\left(Y\right)\right]\int_0^{\epsilon} f(x) dx,
\end{align*}
which shows the assertion.
\end{proof}

\section{Properties of the Adaptive Euler-Maruyama Scheme}
\label{prop:Euler}

\subsection{General properties} 
Recall that our  Euler-Maruyama scheme is given by \begin{align} \label{euler_anf_2}
  \tau_0 = 0, \qquad \appxy_0 &=x \in \mathbb{R}^d,
 \end{align}
 and
 \begin{align} \label{euler_it_2}
 \tau_{k+1}&= \tau_{k} + h(\appxy_{\tau_k},\delta), \quad \appxy_{t} = \appxy_{\tau_k}+ \mu(\appxy_{\tau_k})(t-\tau_k) + \sigma(\appxy_{\tau_k})(W_{t}-W_{\tau_k}), \,\, t\in(\tau_k,\tau_{k+1}],
\end{align}
with $ k\in \N_0$, and $h\colon\R^d\times (0,1)\to(0,1)$,
\begin{align}\label{eq:step size_2}
  h(x,\delta)=\begin{cases}
    \delta^2,\, & x \in \hypsurf^{\varepsilon_2},\\
   \frac{1}{\|\sigma\|_{\infty,\hypsurf^{\varepsilon_0}}^2}\left( \frac{d(x,\hypsurf)}{ \log (1/\delta)}\right)^2, \, & x \in \hypsurf^{\varepsilon_1} \backslash  \hypsurf^{\varepsilon_2},\\
   \delta,\, & x \notin  \hypsurf^{\varepsilon_1},
  \end{cases}
 \end{align}
where \begin{align} \label{eq:step_bd_2} \varepsilon_1= \|\sigma\|_{\infty,\hypsurf^{\varepsilon_0}}\log(1/\delta)\sqrt{\delta}, \qquad \varepsilon_2=\|\sigma\|_{\infty,\hypsurf^{\varepsilon_0}}  \log(1/\delta)\delta. \end{align}
Moreover, set $\et=\max\{\tau_k\colon \tau_k \le t\}$.\\

Figure \ref{fig:step sizeregimes} illustrates the different step size regimes. The solid line is the set of discontinuities $\hypsurf$ of the drift,
the area between the two dashed lines is $\hypsurf^{\varepsilon_2}$, and the area between the dotted lines is $\hypsurf^{\varepsilon_1}$.
\begin{figure}
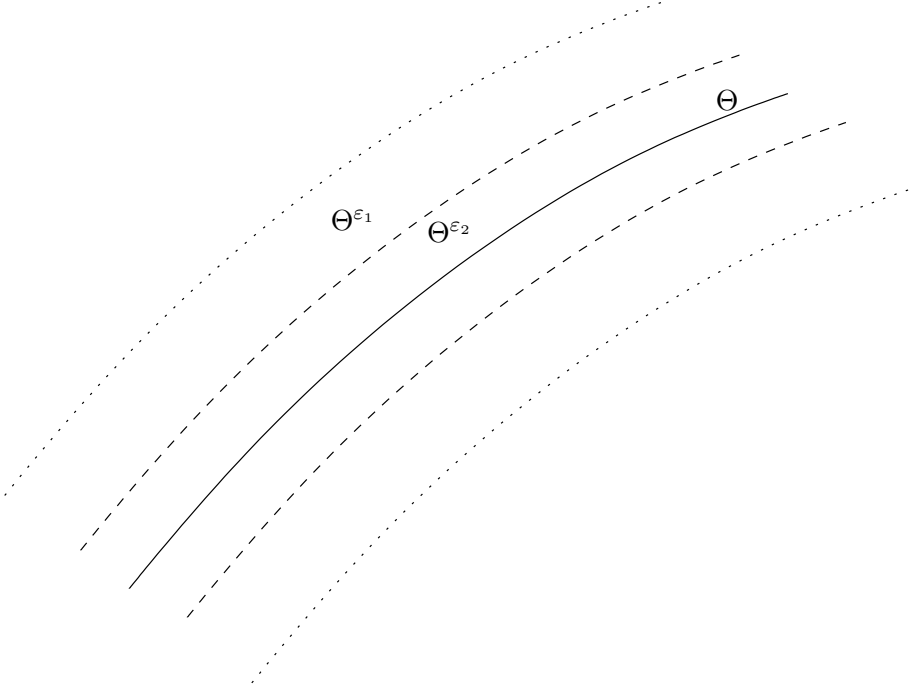
\caption{The three step size regimes.\label{fig:step sizeregimes}}
\end{figure}

\begin{framework}\label{framework}
We assume that
\begin{enumerate}[(i)]
\setlength{\itemsep}{0em}
\item Assumption \ref{ass:existence} holds,
\item the step size $\delta$ is sufficiently small such that
$$ \varepsilon_1= \|\sigma\|_{\infty,\hypsurf^{\varepsilon_0}} \log(1/\delta)\sqrt{\delta} < \varepsilon_0/4,$$
\item the constant $c$ from equation \eqref{eq:varphi} is sufficiently small, see \cite[Lemma 3.18]{sz2016b}, in particular, $c<\varepsilon_0$, such that $G$ is globally invertible.
\end{enumerate}
\end{framework}

For fixed $\delta\in (0,1)$ define the mapping
$$ \Phi:\mathbb{R}^d \times C([0,\infty); \mathbb{R}^d) \rightarrow C([0,\infty); \mathbb{R}^d), \qquad \Phi(x,W)=(X^{h}_s)_{s \geq 0}. $$
Note that by construction $\Phi$  is  $\mathcal{B}(\mathbb{R}^d \times C([0,\infty); \mathbb{R}^d))- \mathcal{B}( C([0,\infty); \mathbb{R}^d))$ measurable.
Also by construction our discretization points $\tau_k$  are stopping times, i.e.~they satisfy
$$ \{ \tau_k \leq  u\} \in \F_u, \quad
u \geq 0, \quad k\in\N_0.$$
This can be shown by induction since $\tau_{k+1}$ is $\appx_{\tau_k}$ measurable and $X_0^h$ is deterministic. The strong Markov property of Brownian motion then implies that  for all $k\in\N_0$ the process $$(W^{\tau_k}_{t})_{t \geq 0}=(W_{t+ \tau_{k}}-W_{\tau_k})_{t \geq 0}$$ is again a Brownian motion and independent of $X_{\tau_k}$.

\begin{lemma}\label{markov} Assume Framework \ref{framework},
let $F \in  \mathcal{B}(C([0,\infty); \R^d))$, and $k \in \mathbb{N}_0$. Then we have
 $$    \P( (\appx_{t+\tau_k})_{t \geq 0} \in F | \appx_{\tau_k}=y) =\P(\Phi(y,W) \in F) \quad  \text{for} \quad  \P^{X_{\tau_k}}\text{-almost all } y \in \mathbb{R}^d.   $$
 \end{lemma}
 \begin{proof}
 Using $\Phi$ and $W^{\tau_k}$ we can write
 \begin{align*}  \appx_{t+\tau_k} &= \appx_{\tau_k} +  \mu(\appx_{\tau_k})t + \sigma(\appx_{\tau_k}) W_{t}^{\tau_k}
  =\Phi(\appx_{\tau_k},W^{\tau_k})(t), \quad t \in [0,\tau_{k+1}-\tau_k].
  \end{align*} Proceeding iteratively
 we obtain
 $$(\appx_{t+\tau_k})_{t \geq 0}=(\Phi(\appx_{\tau_k},W^{\tau_k}))_{t \geq 0}. $$
 Since $W^{\tau_k}$ is a Brownian motion which is independent of $\appx_{\tau_k}$, the factorization Lemma for conditional expectations concludes the proof. 
 \end{proof}

\begin{lemma} \label{msq-appx} Assume Framework \ref{framework} and let $T>0$. For any $p\geq 2$ 
there exist constants $C_p,K_p>0$ such that
$$ \E \Big[ \sup_{t \in [0,T]} \|\appx_t  \|^p   \Big] \leq  C_p    $$  and
$$ \E [ \|\appx_t - \appx_s \|^p ]  \leq K_p \cdot |t-s|^{p/2}, \quad 0 \leq s  <t \leq T. $$
\end{lemma}
\begin{proof}

As preparation for the first statement note that
 $W_{\tau_{k+1}}-W_{\tau_k}$ is independent of $ \appx_{\tau_{k}}$ and
satisfies
$$  \E \left[\| W_{\tau_{k+1}}-W_{\tau_k}\|^p\right] =   \int_{\mathbb{R}}   \E \left[\| W^{\tau_k}_{h(y,\delta)}\|^p\right] \, d \P^{\appx_{\tau_k}}(y)  \leq
 \kappa_{p} \cdot \delta^{p/2}$$
for some constant $ \kappa_p >0$. Clearly, we have
$$  \| \appx_{\tau_{k+1}} \|^p \leq 3^{p-1} \|  \appx_{\tau_{k}}\|^p +   3^{p-1}  \| \mu( \appx_{\tau_{k}})\|^p  \delta^p
  +   3^{p-1}   \| \sigma(\appx_{\tau_k}) \|^p \| W_{\tau_{k+1}}-W_{\tau_k}  \|^p.$$
Hence, the preparations together with the linear growth of $\mu$ and $\sigma$ imply the existence of a constant $c_1 >0$ independent of $k$ such 
that $$  \E \left[\|\appx_{\tau_{k+1}} \|^p\right] \leq c_1 \left(1+ \E \left[\|  \appx_{\tau_{k}}\|^p\right]\right), \quad k\in\N_0. $$
Since $X_0=x$ it follows that
\begin{align}   \label{finite_m_1} \E \left[\| \appx_{\tau_{k}} \|^p \right]< \infty, \qquad k\in\N_0.   \end{align}

We also have
$$  \sup_{t \in [0,s]} \|\appx_t\|^p \leq 3^{p-1} \|x\|^p +   3^{p-1}  \sup_{t \in [0,s]} \left \| \int_0^t  \mu(\appx_{\eu}) du  \right\|^p +   3^{p-1}  \sup_{t \in [0,s]} \left\| \int_0^t  \sigma(\appx_{\eu}) dW_u  \right\|^p, \quad s \in [0,T].$$
The Cauchy-Schwartz inequality for the Riemann-integral, the Burkholder-Davis-Gundy inequality for the It\=o-integral, and the linear  growth condition on $\mu$ and $\sigma$ ensure the existence of a constant $c_2>0$, which depends on $T>0$, $\mu$, $\sigma$, $p$, and $x$, such that
\begin{align} \label{eq_moment_2} \E \Big[ \sup_{t \in [0,s]} \|\appx_t\|^p \Big] \leq c_2 \left(1 +  \int_0^s  \E \left[\| \appx_{\eu} \|^p \right] du \right), \quad s \in [0,T]. \end{align}
By  \eqref{finite_m_1} we now obtain that
\begin{align}   \label{finite_m_2}  \E \Big[ \sup_{t \in [0,T]} \|\appx_t\|^p \l \Big] < \infty .   \end{align}

Equation \eqref{eq_moment_2} also yields
$$  \E  \Big[ \sup_{t \in [0,s]} \|\appx_t  \|^p  \Big]  \leq c_2  + c_2 \int_0^s   \E  \Big[ \sup_{t \in [0,u]} \|\appx_t  \|^p \Big] du.$$
Since  \eqref{finite_m_2} holds,   Gronwall's Lemma now yields the first assertion.\\

For the second statement note that
$$ \| \appx_t - \appx_s\|^{p} \leq  2^{p-1 } \left \| \int_s^t \mu(\appx_{\underline{u}}) du \right \|^p+  2^{p-1 } \left \| \int_s^t \sigma(\appx_{\underline{u}}) dW_u \right \|^p, \quad s,t \in [0,T].$$
The Cauchy-Schwartz inequality,  the Burkholder-Davis-Gundy inequality, and the linear growth condition of the coefficients together with the first assertion yield the statement.
\end{proof}

Applying Theorem \ref{th:occtime}  to our Euler-Maruyama scheme we obtain:
\begin{lemma}\label{it:inner_cont} Assume Framework
\ref{framework} and let $T>0$. Moreover, let   $\varepsilon < \varepsilon_0/2$ and $f\colon[0, \infty) \rightarrow [0,\infty)$ be a measurable function.
Then, there exists a constant $C>0$ such that
$$  \E \left[ \int_0^T f(d(\appx_t,\hypsurf)) \1_{ \{\appx_t \in \hypsurf^{\varepsilon} \} } dt \right] \leq C \, \int_{0}^{\varepsilon} f(x) dx.$$
In particular, we have
\begin{align*}  
\int_{0}^{T} \P( \appx_{t} \in \Theta^{\varepsilon}  ) dt  \le  C \cdot \varepsilon.
\end{align*}
\end{lemma}

\subsection{Exit time estimates for the {adaptive} Euler-Maruyama scheme}

In this subsection we present exit time estimates for our Euler-Maruyama scheme. For this, we require the following Lemma.
\begin{lemma}\label{max_bm}
 There exists a constant $C_{tail}>0$ such that
 \begin{align*}
 \P \left( \sup_{0\le s\le t}\|W_s\| \ge \varepsilon \right) \le C_{tail} \cdot \exp\left(- \frac{\varepsilon}{\sqrt{t}} \right), \quad \varepsilon>0, \quad t\ge0.
\end{align*}
\end{lemma}
\begin{proof} 
Using the scaling property of Brownian motion and applying Doob's submartingale inequality we obtain
 \begin{align*}
& \P \left(  \sup_{0\le s\le t}\|W_s\| \geq \varepsilon \right)
 =  \P \left(  \sup_{0\le s\le 1}\|W_s\| \geq \frac{\varepsilon}{\sqrt{t}}  \right)
 \\ &= \P \left(  \exp\left(\sup_{0\le s\le 1}\|W_s\| \right)\geq \exp \Big( \frac{\varepsilon}{\sqrt{t}} \Big) \right)
  \le \E \left[ \exp(\|W_1\|) \right] \exp \left(-\frac{\varepsilon}{\sqrt{t}} \right).
\end{align*}
\end{proof}

The next lemma controls the probabilities that  the Euler-Maruyama scheme has increments that are 
relatively large compared to its distance from $\hypsurf$.

\begin{lemma}\label{lem:regimes} Assume Framework \ref{framework} and let $T>0$. Then there exists a constant $C>0$, independent of $\delta$, such that
\begin{enumerate}[$(i)$]  
\item\label{it:inner} \quad $\int_0^T \mathbb{P}\big(\appx_t \notin \Theta^{2\varepsilon_2} ; \appx_{\et} \in \hypsurf^{\varepsilon_2} \big) dt  \leq C \cdot \delta$,
\item\label{it:medium} \quad $\int_0^T \mathbb{P}\big( \|\appx_t- \appx_{\et}\| \geq d(\appx_{\et}, \Theta)  ; \appx_{\et} \in \hypsurf^{\varepsilon_1} \setminus  \hypsurf^{\varepsilon_2} \big) dt  \leq C  \cdot \delta$,
\item\label{it:outer} \quad  $\int_0^T \mathbb{P}\big( \|\appx_t- \appx_{\et}\| \geq \varepsilon_1  ; \appx_{\et} \in \hypsurf^{\varepsilon_0}   \setminus  \hypsurf^{\varepsilon_1} \big) dt  \leq C \cdot \delta$.
\end{enumerate}
\end{lemma}

\begin{proof}
(i) \, Note that
$$  \left \{ \omega \in \Omega: \appx_t(\omega)\notin \Theta^{2\varepsilon_2} ; \appx_{\et}(\omega) \in \hypsurf^{\varepsilon_2}  \right \} \subseteq    \left \{ \omega \in \Omega: \|\appx_t(\omega)-  \appx_{\et}(\omega)\| \geq \varepsilon_2 ; \appx_{\et}(\omega) \in \hypsurf^{\varepsilon_2}  \right \} \, $$
and that
\begin{align*}
&\int_0^T \P \big( \|\appx_t-  \appx_{\et}\| \geq \varepsilon_2 ;\appx_{\et} \in \hypsurf^{\varepsilon_2} \big) dt   \leq
 \int_0^T \P \left( \sup_{t \in [\et, \et+h(\appx_{\et},\delta)]} \|\appx_t-  \appx_{\et}\| \geq \varepsilon_2 ; \appx_{\et} \in \hypsurf^{\varepsilon_2} \right) dt\\
 & \qquad =\int_0^T \P \left( \sup_{t \in [\et, \et+h(\appx_{\et},\delta)]} \|\appx_t-  \appx_{\et}\| \geq \varepsilon_2 \Big{|} \appx_{\et} \in \hypsurf^{\varepsilon_2} \right) \P(\appx_{\et} \in \hypsurf^{\varepsilon_2}) dt.
\end{align*}
Here we set the value of the above conditional probability to zero, if $ \P(\appx_{\et} \in \hypsurf^{\varepsilon_2})=0$.
By Lemma  \ref{markov} we obtain that
\begin{align*}
&  \int_0^T \P \left( \sup_{t \in [\et, \et+h(\appx_{\et},\delta)]} \|\appx_t-  \appx_{\et}\| \geq \varepsilon_2 \Big{|} \appx_{\et} \in \hypsurf^{\varepsilon_2} \right) \P(\appx_{\et} \in \hypsurf^{\varepsilon_2}) dt
  \\ & \qquad =
  \int_0^T\int_{\hypsurf^{\varepsilon_2}} \P \left( \sup_{t \in [0, \delta^2]} \|\Phi(y,W^{\et})(t)- y\| \geq \varepsilon_2 \right) d\, \P^{\appx_{\et}}(y)  dt
    \\ & \qquad \le
      \int_0^T\int_{\hypsurf^{\varepsilon_2}} \P \left( \sup_{t \in [0,1]} \|W_t\| \geq \frac{\varepsilon_2- \delta^2 \|\mu\|_{\infty,\hypsurf^{\varepsilon_0}} }{\delta\|\sigma\|_{\infty,\hypsurf^{\varepsilon_0}}} \right) d\,  \P^{\appx_{\et}}(y)  dt.
\end{align*}
Recall that $\varepsilon_2=\|\sigma\|_{\infty,\hypsurf^{\varepsilon_0}}   \log(1/\delta)\delta$, and hence
\begin{align}\label{eq:lemma-help1}
 \frac{\varepsilon_2- \delta^2 \|\mu\|_{\infty,\hypsurf^{\varepsilon_0}} }{ \delta  \|\sigma\|_{\infty,\hypsurf^{\varepsilon_0}}} =-\log(\delta)- \frac{\|\mu\|_{\infty,\hypsurf^{\varepsilon_0}}}{\|\sigma\|_{\infty,\hypsurf^{\varepsilon_0}}}\delta \geq - \log(\delta)- \frac{\|\mu\|_{\infty,\hypsurf^{\varepsilon_0}}}{\|\sigma\|_{\infty,\hypsurf^{\varepsilon_0}}}.
\end{align}
An application of Lemma \ref{max_bm} together with \eqref{eq:lemma-help1} now yields
\begin{align*}
     & \int_0^T\int_{\hypsurf^{\varepsilon_2}} \P \left( \sup_{t \in [0,1]} \|W_t\| \geq \frac{\varepsilon_2- \delta^2 \|\mu\|_{\infty,\hypsurf^{\varepsilon_0}} }{\delta\|\sigma\|_{\infty,\hypsurf^{\varepsilon_0}}} \right)  \,  \P^{\appx_{\et}}(y) dt\\
     &  \qquad \le C_{\textrm{tail}} \int_0^T \exp\left( \log(\delta)\right) \exp \left( \frac{\|\mu\|_{\infty,\hypsurf^{\varepsilon_0}}}{\| \sigma \|_{\infty,\hypsurf^{\varepsilon_0}}} \right) \int_{\hypsurf^{\varepsilon_2}} d\,  \P^{\appx_{\et}}(y) dt
     \leq  C \cdot \delta,
\end{align*}
with $$C=C_{\textrm{tail}} T \exp \left( \frac{\|\mu\|_{\infty,\hypsurf^{\varepsilon_0}}}{\| \sigma \|_{\infty,\hypsurf^{\varepsilon_0}}} \right).$$

(ii)\, Observe that
\begin{align*}
& \int_0^T \mathbb{P}\big( \|\appx_t- \appx_{\et}\| \geq d(\appx_{\et}, \Theta)  ; \appx_{\et} \in \hypsurf^{\varepsilon_1} \setminus  \hypsurf^{\varepsilon_2} \big) dt  \\ 
& \qquad \leq 
\int_0^T \mathbb{P}\left(\sup_{t \in [\et, \et + h(\appx_{\et},\delta)] }\|\appx_t- \appx_{\et}\| \geq d(\appx_{\et}, \Theta)  ; \appx_{\et} \in \hypsurf^{\varepsilon_1} \setminus  \hypsurf^{\varepsilon_2} \right) dt\\ 
& \qquad =
\int_0^T \mathbb{P}\left(\sup_{t \in [\et, \et + h(\appx_{\et},\delta)] }\|\appx_t- \appx_{\et}\| \geq d(\appx_{\et}, \Theta)  \Big{|} \appx_{\et} \in \hypsurf^{\varepsilon_1} \setminus  \hypsurf^{\varepsilon_2} \right) \P( \appx_{\et} \in \hypsurf^{\varepsilon_1} \setminus  \hypsurf^{\varepsilon_2} )dt.
\end{align*}
Using again Lemma \ref{markov} we obtain
\begin{align*}
&\int_0^T \mathbb{P}\left(\sup_{t \in [\et, \et + h(\appx_{\et},\delta)] }\|\appx_t- \appx_{\et}\| \geq d(\appx_{\et}, \Theta)  \Big{|} \appx_{\et} \in \hypsurf^{\varepsilon_1} \setminus  \hypsurf^{\varepsilon_2} \right) \P( \appx_{\et} \in \hypsurf^{\varepsilon_1} \setminus  \hypsurf^{\varepsilon_2} )dt
 \\ & \qquad =
  \int_0^T\int_{ \hypsurf^{\varepsilon_1} \setminus  \hypsurf^{\varepsilon_2}} \P \left( \sup_{t \in [0, h(y, \delta)]}\|\Phi(y,W^{\et})(t)-y\| \geq d(y,\Theta) \right)  d\,  \P^{\appx_{\et}}(y)  dt
  \\ & \qquad \le
      \int_0^T\int_{\hypsurf^{\varepsilon_1} \setminus  \hypsurf^{\varepsilon_2}} \P \left( \sup_{t \in [0,1]} \|W_t\| \geq \frac{d(y,\hypsurf)-\|\mu\|_{\infty,\hypsurf^{\varepsilon_0}} h(y,\delta)}{h(y,\delta)^{1/2} \|\sigma\|_{\infty,\hypsurf^{\varepsilon_0}}} \right)  d\,  \P^{\appx_{\et}}(y)  dt.
\end{align*}
Since $$h(x,\delta)=
       \frac{1}{\|\sigma\|_{\infty,\hypsurf^{\varepsilon_0}}^2} \left( \frac{d(x,\hypsurf)}{ \log (1/\delta)}\right)^2, \qquad x \in \hypsurf^{\varepsilon_1} \backslash  \hypsurf^{\varepsilon_2},$$ we arrive at
\begin{align}\label{eq:lemma-help2}
 \frac{d(y,\hypsurf)-\|\mu\|_{\infty,\hypsurf^{\varepsilon_0}} h(y,\delta)}{h(y,\delta)^{1/2} \|\sigma\|_{\infty,\hypsurf^{\varepsilon_0}}} =-\log(\delta)- \frac{\|\mu\|_{\infty,\hypsurf^{\varepsilon_0}}}{\|\sigma\|_{\infty,\hypsurf^{\varepsilon_0}}}h(y,\delta)^{1/2}\ge-\log(\delta)- \frac{\|\mu\|_{\infty,\hypsurf^{\varepsilon_0}}}{\|\sigma\|_{\infty,\hypsurf^{\varepsilon_0}}}.
\end{align}
Lemma \ref{max_bm} together with \eqref{eq:lemma-help2} yields that
\begin{align*}
     & \int_0^T \int_{ \hypsurf^{\varepsilon_1} \setminus  \hypsurf^{\varepsilon_2}} \P \left( \sup_{t \in [0,1]} \|W_t\| \geq \frac{d(y,\hypsurf)-\|\mu\|_{\infty,\hypsurf^{\varepsilon_0}} h(y,\delta)}{h(y,\delta)^{1/2} \|\sigma\|_{\infty,\hypsurf^{\varepsilon_0}}} \right)  d\,  \P^{\appx_{\et}}(y) dt\
     \le C \cdot \delta.
\end{align*}

(iii)\, This can be shown along the same lines as (ii) taking into account that the Euler-Maruyama scheme under the condition $\appx_{\et} \notin \hypsurf^{\varepsilon_1}$ has step size $\delta$ and that $\varepsilon_1= \| \sigma \|_{\infty,\hypsurf^{\varepsilon_0}} \log(1/\delta) \sqrt{\delta}.$
\end{proof}

Using the previous lemmas we obtain:
\begin{lemma} \label{it:discrete} Assume Framework \ref{framework} and let $T>0$. There exists a constant $C>0$ such that
$$\int_0^T \P\left( \appx_{\et} \in \Theta^{\varepsilon_2}  \right)dt  \leq  \  C\cdot (1+\log(1/\delta)) \delta ,$$
and
$$\int_0^T \P\left( \appx_{\et} \in \Theta^{\varepsilon_1} \setminus \Theta^{\varepsilon_2}  \right)dt  \leq  \  C \cdot (1+\log(1/\delta)) \sqrt{\delta} .$$
\end{lemma}
 
\begin{proof}
To establish the first estimate we write 
\begin{align*}
\int_0^T \P\left(\appx_{\et} \in \Theta^{\varepsilon_2}  \right)dt  & =  \int_0^T \P\left( \appx_{\et} \in \Theta^{\varepsilon_2}; \appx_{t} \notin \Theta^{2 \varepsilon_2}  \right)dt +\int_0^T \P\left(\appx_{\et} \in \Theta^{\varepsilon_2};  \appx_{t} \in \Theta^{2 \varepsilon_2} \right)dt
 \\ & \leq   \int_{0}^T \P\left( \appx_{\et} \in \Theta^{\varepsilon_2};\appx_{t} \notin \Theta^{2 \varepsilon_2} \right)dt+ \int_0^T \P\left( \appx_{t} \in \Theta^{2 \varepsilon_2} \right) dt .
 \end{align*}
We conclude by Lemma \ref{lem:regimes}\eqref{it:inner} and Lemma $\ref{it:inner_cont}$.
Analogously  the second estimate follows from
\begin{align*}
& \int_0^T \P\left(\appx_{\et} \in  \Theta^{\varepsilon_1} \setminus \Theta^{\varepsilon_2}  \right)dt  \\ &  \qquad  =  \int_0^T \P\left( \appx_{\et} \in \Theta^{\varepsilon_1} \setminus \Theta^{\varepsilon_2}; \appx_{t} \notin \Theta^{2 \varepsilon_1}  \right)dt +\int_0^T \P\left(\appx_{\et} \in \Theta^{\varepsilon_1} \setminus \Theta^{\varepsilon_2};  \appx_{t} \in \Theta^{2 \varepsilon_1} \right)dt
 \\ &  \qquad \leq   
 \int_0^T \mathbb{P}\left( \|\appx_t- \appx_{\et}\| \geq d(\appx_{\et}, \Theta)  ; \appx_{\et} \in \hypsurf^{\varepsilon_1} \setminus  \hypsurf^{\varepsilon_2} \right)
 + \int_0^T \P\left( \appx_{t} \in \Theta^{2 \varepsilon_1} \right) dt .
 \end{align*}
We conclude  by Lemma \ref{lem:regimes}\eqref{it:medium} and Lemma $\ref{it:inner_cont}$.

\end{proof}
Finally we prove the following refinement of Lemma \ref{lem:regimes}\eqref{it:medium}:
\begin{lemma}\label{lem:regimes_extra} Assume Framework \ref{framework}, let $T> 0$ and $\alpha \in (0,1]$. There exists a  constant $C>0$ such that
$$ \int_0^T \mathbb{P}\big( \|\appx_t- \appx_{\et}\| \geq \alpha \cdot d(\appx_{\et}, \Theta)  ; \appx_{\et} \in \hypsurf^{\varepsilon_1} \setminus  \hypsurf^{\varepsilon_2} \big) dt  \leq C \cdot (1+ \log(1/\delta)) \delta^{\alpha + \frac{1}{2}}.$$
\end{lemma}

\begin{proof} Here we follow the same lines as in the proof of  Lemma \ref{lem:regimes}\eqref{it:medium} to obtain
\begin{align*}
&\int_0^T \mathbb{P}\left(\|\appx_t- \appx_{\et}\| \geq  \alpha \cdot d(\appx_{\et}, \Theta)  ; \appx_{\et} \in \hypsurf^{\varepsilon_1} \setminus  \hypsurf^{\varepsilon_2} \right) dt
\\ & \leq       \int_0^T\int_{\hypsurf^{\varepsilon_1} \setminus  \hypsurf^{\varepsilon_2}} \P \left( \sup_{t \in [0,1]} \|W_t\| \geq -\alpha \log(\delta) - \frac{\|\mu\|_{\infty,\hypsurf^{\varepsilon_0}}}{\|\sigma\|_{\infty,\hypsurf^{\varepsilon_0}}} \right)  d\,  \P^{\appx_{\et}}(y) dt.
\end{align*}
Lemma \ref{max_bm} now yields that
\begin{align*}
&\int_0^T \mathbb{P}\left(\|\appx_t- \appx_{\et}\| \geq \alpha \cdot d(\appx_{\et}, \Theta)  ; \appx_{\et} \in \hypsurf^{\varepsilon_1} \setminus  \hypsurf^{\varepsilon_2} \right)dt
\\ & \leq    C_{\textrm{tail}}  \exp \left( \frac{\|\mu\|_{\infty,\hypsurf^{\varepsilon_0}}}{\| \sigma \|_{\infty,\hypsurf^{\varepsilon_0}}} \right) \delta^{\alpha} \int_0^T \P( \appx_{\et} \in \hypsurf^{\varepsilon_1} \setminus  \hypsurf^{\varepsilon_2}) dt.
\end{align*}
The second case from Lemma \ref{it:discrete} concludes the proof.
\end{proof}

\section{Convergence Analysis}
\label{sec:convergence}

We are ready to prove the main convergence result.

\begin{theorem}\label{thm:conv}
Assume Framework \ref{framework} and let $T>0$.
 Moreover, let $h\colon\R^d\times (0,1)\to (0,1)$ be given by  \eqref{eq:step size_2} and \eqref{eq:step_bd_2}, and let $X^h$ be given by \eqref{euler_anf_2} and \eqref{euler_it_2}. Then there exists a constant $C_{\textrm{rmse}}>0$ such that
 \begin{align*}
 \E \left[ \sup_{0\le t\le T}  \|X_t-\appx_t \|^2 \right] \le C_{\textrm{rmse}}^2 \cdot (1+\log(1/\delta)) \delta.
\end{align*}
\end{theorem}

\begin{proof}
The proof will be split into three steps. The first one uses the transformation $G$, the second the exit probability estimates from the previous section and the last step is a Gronwall argument.
We denote constants independent of $\delta$ by $c_1,c_2,\dots$.
\\

\textbf{Step 1:}
Here we follow \cite[Theorem 3.1]{sz2017c}, where the convergence proof is done by means of the transformation $G$, see  Subsection \ref{def_g_etal}. We define a process $Z=(Z_t)_{t\ge 0}$ by $Z_t=G(X_t)$. It solves
\begin{align}\label{eq:SDEtransf_2}
 dZ_t = \mu_G(Z_t) dt+\sigma_G (Z_t) dW_t, \quad t \geq 0, \qquad Z_0=G(x),
\end{align}
where $\mu_G$, $\sigma_G$ are given by \eqref{musig_G} and are globally Lipschitz by Lemma \ref{lem:propG_lip}. Moreover, $G$ and $G^{-1}$ are globally Lipschitz by Lemma \ref{lem:propG}(\ref{it:GLip}). So we obtain
\begin{align}\label{eq:est-lip}
 \Big( \E \Big[ \sup_{0\le t\le T} \|X_t-\appx_t \|^2 \Big] \Big)^{1/2}
\le L_{G^{-1}}  \Big( \E \Big[ \sup_{0\le t\le T} \|Z_t-G (\appx_t ) \|^2 \Big] \Big)^{1/2}.
\end{align}
Now, denote by $\appz$ the Euler-Maruyama approximation of the process $Z$ based on the step sizing function $h$ given by  \eqref{eq:step size_2} and \eqref{eq:step_bd_2}. Using this scheme we can split the error as follows:
\begin{align}\label{eq:est-dreieck}
  & \Big(\E \Big[ \sup_{0\le t\le T} \|Z_t-G (\appx_t ) \|^2 \Big] \Big)^{1/2}
 \\ & \quad \le  \Big( \E \Big[ \sup_{0\le t\le T} \|Z_t- \appz_t \|^2  \Big] \Big)^{1/2}
+ \Big( \E  \Big[\sup_{0\le t\le T} \|\appz_t-G (\appx_t ) \|^2 \Big] \Big)^{1/2}.  \nonumber
\end{align}
Since the maximum step size of $\appz$ is $\delta$,  there exists a constant $c_1>0$ such that
\begin{align}\label{eq:est-euler}
\E \Big[\sup_{0\le t\le T}\|Z_t- \appz_t\|^2 \Big]\le c_1 \cdot \delta.
\end{align}
This can be shown by a simple modification of the error analysis of the Euler-Maruyama scheme for deterministic non-equidistant discretizations as, e.g., in 
 \cite[Theorem 10.2.2]{kloeden1992}. 

The second error term in \eqref{eq:est-dreieck} is the difference between the transformation applied to the time continuous Euler-Maruyama approximation of $X$ defined in \eqref{eq:euler-timecont} and the Euler-Maruyama approximation of the transformed process $Z$ defined in \eqref{eq:SDEtransf}.
For all $\tau\in[0,T]$ set
\begin{align*}
u(\tau):=\E \Big[ \sup_{0\le t\le \tau}
\|G(\appx_{t})-\appz_{t}\|^2 \Big] .
\end{align*}
The Lipschitz continuity of $G$, $\mu_G$, and $\sigma_G$ and Lemma  \ref{msq-appx} imply that $\sup_{\tau \in [0,T]} u(\tau) < \infty$.
For all $x_1,x_2\in \R^d$ define $$\nu(x_1,x_2):=G'(x_1)\mu(x_2)+\frac{1}{2}\tr (\sigma(x_2)^\top G''(x_1)\sigma(x_2)),$$ and notice that $\nu(x,x)=\mu_G(G(x))$,
$\sigma_G(G(x))=G'(x) \sigma(x) $.

By It\=o's formula we have
\begin{align*}
G(\appx_{t})=G(\appx_{0})+\int_0^{t}\nu(\appx_s,\appx_{\es}) ds+\int_0^{t}G'(\appx_s)\sigma(\appx_{\es})dW_s, \quad t \in [0,T].
\end{align*}
With this, we get that
\begin{align*}
 u(\tau)
&=\E\left[ \sup_{0\le t\le \tau}\left\|\int_0^{t} \nu(\appx_s,\appx_{\es})ds
+\int_0^{t}G'(\appx_s)\sigma(\appx_{\es})dW_s
-\int_0^{t}\mu_G (\appz_{\es}) ds-\int_0^{t} \sigma_G (\appz_{\es}) dW_s\right\|^2\right] \\
&\,\, \le 4 \E \left[ \sup_{0\le t\le \tau} \left\|\int_0^{t} \left( \nu(\appx_s,\appx_{\es}) 
-\nu(\appx_{\es},\appx_{\es}) \right) ds\right\|^2\right] \\ & \qquad
+ 4 \E \left[ \sup_{0\le t\le \tau} \left\|\int_0^{t}\left(G'(\appx_s)\sigma(\appx_{\es})-G'(\appx_{\es})\sigma(\appx_{\es})\right)dW_s\right\|^2  \right]
\\ 
&\qquad + 4 \E \left[ \sup_{0\le t\le \tau} \left\|\int_0^{t}\left(\mu_G(G(\appx_{\es}))-\mu_G(\appz_{\es})\right) ds \right\|^2\right] \\ 
& \qquad +4 \E\left[  \sup_{0\le t\le \tau} \left\|\int_0^{t}\left( \sigma_G(G(\appx_{\es}))- \sigma_G(\appz_{\es})\right)dW_s\right\|^2\right].
\end{align*}

With the Cauchy-Schwarz inequality and the $d$-dimensional Burkholder-Davis-Gundy inequality, see, e.g., \cite[Theorem III.3.28]{karatzas1991} or \cite[Lemma 3.7]{hutzenthaler2012}, we obtain
\begin{align}
u(\tau) \leq 4T\,E_1(\tau) +8d\,E_2(\tau)+4T\,E_3(\tau)+8d\,E_4(\tau), \label{eq:g-z}
\end{align}
with
\begin{equation}\label{eq:Es}
\begin{aligned}
E_1(\tau) &= \E\left[  \int_0^{\tau}\left\|\nu(\appx_s,\appx_{\es}) 
-\nu(\appx_{\es},\appx_{\es})\right\|^2 ds \right], \\E_2(\tau) & =
\E \left[ \int_0^{\tau}\left\|G'(\appx_s)\sigma(\appx_{\es})-G'(\appx_{\es})\sigma(\appx_{\es})\right\|^2 ds\right],
\\
E_3(\tau) &=\E \left[ \int_0^{\tau}\left\|\mu_G(G(\appx_{\es}))- \mu_G(\appz_{\es}) \right\|^2 ds\right],\\
E_4(\tau)& =\E\left[  \int_0^{\tau}\left\|\sigma_G(G(\appx_{\es}))- \sigma_G(\appz_{\es})\right\|^2 ds\right].
\end{aligned}
\end{equation}
\textbf{Step 2:} Now we estimate the above error terms. 
For $E_1$, using the linear growth property of $\mu$ and $\sigma$ and the  properties of $G$ we have
\begin{align}\label{eq:est-E1-help}
\left\|\nu(x_1,x_2)-\nu(x_2,x_2)\right\|^2\le
\begin{cases}
K_1 \cdot (1+\|x_2\|^4) \cdot \|x_1-x_2\|^2,& \qquad   \,\, \|x_1 -x_2\|=\rho(x_1,x_2), \\
K_2  \cdot (1+\|x_2\|^4), &  \qquad \qquad \qquad \text{otherwise} ,
\end{cases}
\end{align}
with
\begin{align}\label{const_nu}
K_1 &= 2L_{G'}^2 C_{\mu}^2 +\frac{1}{2}L_{G''}^2  C_{\sigma} ^4, \qquad  K_2 = 	4 C_{\mu}^2 \|G'\|^2_\infty +  C_{\sigma}^4 \|G''\|_\infty^2,
\end{align}
where $C_{\mu}>0$ and $C_{\sigma}>0$ are the linear growth constants of the respective coefficients.
Note that $\|x_1 -x_2\|\neq\rho(x_1,x_2)$ means that the direct connection between $x_1$ and $x_2$ passes $\hypsurf$.
Further,  set
\begin{align} \label{eq:est-E1-help_bound} K_3=\sup \left \{  \|\nu(x_1,x_2) \|^2: \, x_2 \in \hypsurf^{\varepsilon_0}, x_1 \in \mathbb{R}^d \right \}.
\end{align}
The latter quantity is finite due to our assumptions.

We will use the following partitions of 1:
\begin{align*} 1= \1_{\{\|\appx_\es -\appx_s\|=\rho(\appx_\es,\appx_s)\}} + \1_{\{\|\appx_\es -\appx_s\| \neq \rho(\appx_\es,\appx_s)\}}\,
\end{align*} 
and
\begin{align*}
\1_{\{\|\appx_\es -\appx_s\| \neq \rho(\appx_\es,\appx_s)\}}
&=    \1_{\{\|\appx_\es -\appx_s\| \neq \rho(\appx_\es,\appx_s)\}}  \1_{\{\appx_\es \notin \hypsurf^{\varepsilon_0} \}} 
\\ & \qquad +  \1_{\{\|\appx_\es -\appx_s\| \neq \rho(\appx_\es,\appx_s)\}}  \1_{\{\appx_\es \in \hypsurf^{\varepsilon_0}\backslash \hypsurf^{\varepsilon_1} \}} 
\\ & \qquad +  \1_{\{\|\appx_\es -\appx_s\| \neq \rho(\appx_\es,\appx_s)\}}  \1_{\{\appx_\es \in \hypsurf^{\varepsilon_1}\backslash \hypsurf^{\varepsilon_2} \}} 
\\ & \qquad 
+  \1_{\{\|\appx_\es -\appx_s\| \neq \rho(\appx_\es,\appx_s)\}}\1_{\{\appx_\es \in \hypsurf^{\varepsilon_2} \}}\,
\end{align*} for a given $s\in[0,T]$,
i.e.~we split $\Omega$ first into the disjoint events that $\|\appx_\es -\appx_s\| = \rho(\appx_\es,\appx_s)$
or not, and if not,
we split again according to the distance to $\hypsurf$.

\begin{enumerate}[(i)]

\item From \eqref{eq:est-E1-help} we get that
\begin{align*}
&	\E\left[ \int_0^{\tau}  \1_{ \{ \|\appx_\es -\appx_s\|=\rho(\appx_\es,\appx_s) \}} \left\|\nu(\appx_s,\appx_{\es}) 
-\nu(\appx_{\es},\appx_{\es})\right\|^2 ds\right] \\ 
& \leq  K_1 \E \left[\int_0^{\tau} \big(1+\|\appx_{\es}\|^4 \big) \|\appx_\es -\appx_s\|^2 ds\right].
\end{align*}
An application of the Cauchy-Schwarz inequality and Lemma \ref{msq-appx} now yield
\begin{align}\label{eq:E11}
 \E \left[\int_0^{\tau}  \1_{ \{ \|\appx_\es -\appx_s\|=\rho(\appx_\es,\appx_s) \}} \left\|\nu(\appx_s,\appx_{\es}) 
-\nu(\appx_{\es},\appx_{\es})\right\|^2 ds\right]  \leq c_2 \cdot \delta.
\end{align}

\item Now consider the case that $\|\appx_\es -\appx_s\|\neq \rho(\appx_\es,\appx_s) $  and that $\appx_\es$ is more than $\varepsilon_0$ away from $\hypsurf$.
Here \eqref{eq:est-E1-help} gives
\begin{align*}
&	\E \left[\int_0^{\tau}  \1_{ \{ \|\appx_\es -\appx_s\|\neq \rho(\appx_\es,\appx_s) \}} \1_{ \{  \appx_\es \notin \hypsurf^{\varepsilon_0} \} }\left\|\nu(\appx_s,\appx_{\es}) 
-\nu(\appx_{\es},\appx_{\es})\right\|^2 ds\right]\\ 
& \leq  K_2\E\left[ \int_0^{\tau} \big(1+\|\appx_{\es}\|^4 \big) \1_{ \{ \|\appx_\es -\appx_s\|\neq \rho(\appx_\es,\appx_s) \}} \1_{ \{  \appx_\es \notin \hypsurf^{\varepsilon_0} \} }  ds\right].
\end{align*}
Since
\begin{align*}
& \left\{ \|\appx_\es -\appx_s\|\neq \rho(\appx_\es,\appx_s) \right \} \cap \left\{ \appx_\es \notin \hypsurf^{\varepsilon_0} \right\}
 \subseteq    \left\{ \appx_\es \notin \hypsurf^{\varepsilon_0}   \right\}
 \cap \left \{ \|\appx_\es -\appx_s\|\ \geq  \varepsilon_0  \right\},
\end{align*}
the Cauchy-Schwarz inequality together with \eqref{eq:est-E1-help} yields
\begin{align*}
&	\E \left[\int_0^{\tau}  \1_{ \{ \|\appx_\es -\appx_s\|\neq \rho(\appx_\es,\appx_s) \}} \1_{ \{  \appx_\es \notin \hypsurf^{\varepsilon_0} \} }\left\|\nu(\appx_s,\appx_{\es}) 
-\nu(\appx_{\es},\appx_{\es})\right\|^2 ds\right] \\ 
& \leq  K_2 \, \int_0^{\tau} \left( \E \left[ \left|1+\|\appx_{\es}\|^4 \right|^2 \right] \right)^{1/2}  \left(  \P( \|\appx_\es -\appx_s\|\geq \varepsilon_0 ) \right)^{1/2} ds.
\end{align*}
Markov's inequality, i.e.
$$  \P( \|\appx_\es -\appx_s\|\geq \varepsilon_0 )   \leq  \frac{\E [\|\appx_\es -\appx_s\|^4] }{\varepsilon_0^4},  $$
and Lemma  \ref{msq-appx} now give
\begin{align}\label{eq:E12}
&	\E\left[ \int_0^{\tau}  \1_{ \{ \|\appx_\es -\appx_s\|\neq \rho(\appx_\es,\appx_s) \}} \1_{ \{  \appx_\es \notin \hypsurf^{\varepsilon_0} \} }\left\|\nu(\appx_s,\appx_{\es}) 
-\nu(\appx_{\es},\appx_{\es})\right\|^2 ds\right] \leq c_3 \cdot \delta.
\end{align}

\item The next case is that
$\|\appx_\es -\appx_s\|\neq \rho(\appx_\es,\appx_s) $  
 and $\appx_\es$ lies in $ \hypsurf^{\varepsilon_0}\backslash \hypsurf^{\varepsilon_1}$.
Since $$ \nu(\appx_s,\appx_{\es})^2\cdot \1_{ \{  \appx_\es  \in \hypsurf^{\varepsilon_0}\backslash \hypsurf^{\varepsilon_1}  \} } \leq K_3,$$ we obtain
\begin{align*}
& \E\left[ \int_0^{\tau}   \1_{ \{  \appx_\es  \in \hypsurf^{\varepsilon_0}\backslash \hypsurf^{\varepsilon_1}  \} } \1_{ \{  \|\appx_\es -\appx_s\|\neq\rho(\appx_\es,\appx_s) \}  }\left\|\nu(\appx_s,\appx_{\es}) 
-\nu(\appx_{\es},\appx_{\es})\right\|^2 ds\right] \\ 
& \qquad \leq  2K_3 \,\int_0^{\tau} \P({\appx_\es \in \hypsurf^{\varepsilon_0}\backslash \hypsurf^{\varepsilon_1} } ; {\|\appx_\es -\appx_s\|\neq\rho(\appx_\es,\appx_s)}) ds.
\end{align*}
Since
\begin{align*}
& \left\{ \appx_\es \in \hypsurf^{\varepsilon_0}\backslash \hypsurf^{\varepsilon_1}  \right\}
 \cap \left\{ \|\appx_\es -\appx_s\|\neq\rho(\appx_\es,\appx_s) \right\}
 \subseteq    \left\{ \appx_\es \in \hypsurf^{\varepsilon_0}\backslash \hypsurf^{\varepsilon_1}  \right\}
 \cap \left \{ \|\appx_\es -\appx_s\|\ \geq \varepsilon_1  \right\},
\end{align*}
Lemma \ref{lem:regimes}\eqref{it:outer} gives
\begin{align}\label{eq:E13}
& \E\left[ \int_0^{\tau}   \1_{ \{  \appx_\es  \in \hypsurf^{\varepsilon_0}\backslash \hypsurf^{\varepsilon_1}  \} } \1_{ \{  \|\appx_\es -\appx_s\|\neq\rho(\appx_\es,\appx_s) \}  }\left\|\nu(\appx_s,\appx_{\es}) 
-\nu(\appx_{\es},\appx_{\es})\right\|^2 ds\right] \leq c_4 \cdot \delta.
\end{align}

\item
For the next case observe that
\begin{align*}
& \left\{ \appx_\es \in \hypsurf^{\varepsilon_1}\backslash \hypsurf^{\varepsilon_2}  \right\}
 \cap \left\{ \|\appx_\es -\appx_s\|\neq\rho(\appx_\es,\appx_s) \right\}
 \subseteq    \left\{ \appx_\es \in \hypsurf^{\varepsilon_1}\backslash \hypsurf^{\varepsilon_2}  \right\}
 \cap \left \{ \|\appx_\es -\appx_s\|\ \geq d(\appx_\es,\hypsurf)  \right\},
\end{align*}
and so   \eqref{eq:est-E1-help_bound}  and Lemma \ref{lem:regimes}\eqref{it:medium} yield
\begin{equation}\label{eq:E14}
\begin{aligned}
& \E\left[ \int_0^{\tau}   \1_{ \{  \appx_\es  \in \hypsurf^{\varepsilon_1}\backslash \hypsurf^{\varepsilon_2}  \} } \1_{ \{  \|\appx_\es -\appx_s\|\neq\rho(\appx_\es,\appx_s) \}  }\left\|\nu(\appx_s,\appx_{\es}) 
-\nu(\appx_{\es},\appx_{\es})\right\|^2 ds\right] \\ 
& \qquad \leq  2K_3 \,\int_0^{\tau} \P({\appx_\es \in \hypsurf^{\varepsilon_1}\backslash \hypsurf^{\varepsilon_2} } ; {\|\appx_\es -\appx_s\|\neq\rho(\appx_\es,\appx_s)}) ds \\ & \qquad \leq c_5 \cdot \delta.
\end{aligned} 
\end{equation}

\item For the final case,  the boundedness of the coefficients on $\hypsurf^{\varepsilon_0}$ and the fact that 
$$ \left \{\|\appx_\es -\appx_s\| \neq \rho(\appx_\es,\appx_s)  \right \} \cap \left\{\appx_\es \in \hypsurf^{\varepsilon_2}  \right \} \subseteq  \left \{\appx_\es \in \hypsurf^{\varepsilon_2}  \right \} $$
together with the first statement of Lemma \ref{it:discrete} yield that
\begin{equation}\label{eq:E15} \E \left[ \int_0^T     \1_{\{\|\appx_\es -\appx_s\| \neq \rho(\appx_\es,\appx_s)\}}\1_{\{\appx_\es \in \hypsurf^{\varepsilon_2} \}} \left\|\nu(\appx_s,\appx_{\es}) 
-\nu(\appx_{\es},\appx_{\es})\right\|^2  ds \right] \leq c_6 \cdot \log(1/\delta) \delta.
\end{equation}
\end{enumerate}

Combining \eqref{eq:Es} with \eqref{eq:E11}, \eqref{eq:E12}, \eqref{eq:E13}, \eqref{eq:E14}, and \eqref{eq:E15} yields
\begin{align}  E_1(\tau) \leq c_7 \cdot (1+ \log(1/\delta)) \delta. \label{eq:e1} \end{align}

For estimating $E_2$ in \eqref{eq:g-z}, we exploit that $G'$ is globally Lipschitz, $\sigma$ satisfies a linear growth condition, and use Lemma \ref{msq-appx} to obtain
\begin{equation}\label{eq:e2}
\begin{aligned}
E_2(\tau) & \le L_{G'}^2 C_{\sigma}^2 \int_0^{T}  \E  \left[ (1+ \|\appx_{\es} \|^2)   \|\appx_s-\appx_{\es} \|^2 \right] ds   \\
& \le  L_{G'}^2 C_{\sigma}^2 \int_0^{T}  \left( \E   \left[  \big| 1+ \|\appx_{\es} \|^2 \big| ^2 \right] \right)^{1/2}  \left( \E  [ \|\appx_s-\appx_{\es} \|^4] \right)^{1/2} 
ds
\le c_8 \cdot \delta.
\end{aligned}
\end{equation}

For the remaining two terms in \eqref{eq:g-z}
we use the fact that $ \mu_G,  \sigma_G$ are  globally Lipschitz by Lemma \ref{lem:propG_lip}. This gives
\begin{align}\label{eq:e3}
E_3(\tau) &\le L_{\mu_G}^2\int_0^{\tau}\E [\|G(\appx_{\es})-\appz_{\es}\|^2] ds\le L_{\mu_G}^2 \int_0^{\tau} u(s) ds,
\\ \label{eq:e4}
E_4(\tau)& \le L_{\sigma_G}^2\int_0^{\tau}\E[ \|G(\appx_{\es})-\appz_{\es} \|^2] ds\le  L_{\sigma_G}^2\int_0^{\tau}u(s)ds.
\end{align}

{\bf Step 3:} Combining \eqref{eq:g-z} with the estimates \eqref{eq:e1}, \eqref{eq:e2}, \eqref{eq:e3}, and \eqref{eq:e4} we obtain
\begin{align*}
0\le u(\tau) \le c_9 \int_0^\tau u(s) ds + c_{10} \cdot \delta (1+ \log(1/\delta)), \quad  \tau \in [0,T].
\end{align*}
Gronwall's inequality yields
\begin{align}\label{eq:est-gron}
u(\tau)\le c_{10} \exp(c_9 \tau) \cdot \delta (1+ \log(1/\delta)), \quad \tau \in [0,T].
\end{align}

Finally, combining \eqref{eq:est-dreieck} with \eqref{eq:est-euler} and \eqref{eq:est-gron}, and the result with \eqref{eq:est-lip} concludes the proof.

\end{proof}

\begin{remark}
In \cite[Theorem 3.1]{sz2017c} the authors prove strong convergence order $1/4-\epsilon$ for arbitrarily small $\epsilon>0$ of the equidistant Euler-Maruyama scheme under Assumption \ref{ass:existence} and under the additional assumption that the coefficients $\mu$ and $\sigma$ are bounded.
By applying some of the techniques from the proof of Theorem \ref{thm:conv} here, \cite[Theorem 3.1]{sz2017c} can be shown without assuming global boundedness of $\mu$ and $\sigma$.
\end{remark}

\section{Cost Analysis}
\label{sec:cost}

We now turn to  the computational cost  of our step size procedure. 
As mentioned, the computational cost of our method is  proportional to the number of steps, i.e.
\begin{align}\label{eq:cost}
N(h,\delta)= \inf \{k \in \mathbb{N}: \tau_k \geq T \}.
\end{align}
Clearly, we have
$$ N(h,\delta)
\leq 1 + \int_0^T \frac{1}{h(\appx_{\et},\delta)} dt,$$
since
$$    \int_{\tau_k}^{\tau_{k+1}} \frac{1}{h(\appx_{\et},\delta)} dt =1.   $$

\begin{theorem}\label{thm:cost} Assume Framework \ref{framework} and let $T>0$.
 Moreover, let $h\colon\R^d\times (0,1)\to (0,1)$ be given by  \eqref{eq:step size_2} and \eqref{eq:step_bd_2}, and let $X^h$ be given by \eqref{euler_anf_2} and \eqref{euler_it_2}. 
Then there exists a constant $C_{\textrm{cost}}>0$ such that 
$$ \E[ N(h,\delta) ]\leq  C_{\textrm{cost}} \cdot (1+\log(1/\delta))\delta^{-1}.$$
\end{theorem}

\begin{proof}
We denote constants independent of $\delta$ by $c_1,c_2,\dots$.
We have
\begin{equation}  \label{eq:est-N}
\begin{aligned}
 \E [N(h,\delta)]&\le 1+ \E \left[\int_0^T \frac{1}{h(\appx_{\et},\delta)} dt \right] 
\\ &= 1+ I_1 +I_2 +I_3,
\end{aligned}
\end{equation}
with
\begin{equation}\label{eq:Is}
\begin{aligned}
I_1&=\delta^{-2} \int_0^T \P( \appx_{\et} \in \Theta^{\varepsilon_2}) dt,\\
I_2&=\int_0^T \E \left[\frac{1}{h(\appx_{\et}, \delta)} \1_{ \{\appx_{\et}  \in \hypsurf^{\varepsilon_1} \setminus  \hypsurf^{\varepsilon_2} \}} \right] dt  ,\\
I_3&=\delta^{-1} \int_0^T\P( \appx_{\et} \notin \Theta^{\varepsilon_1}) dt .
\end{aligned}
\end{equation}
We also have
 \begin{align} \label{est1} I_3  \leq T \cdot \delta^{-1} ,   \end{align}  
 and by Lemma \ref{it:discrete}
 \begin{align} \label{est2}
I_1\leq c_1 \cdot (1+\log(1/\delta))\delta^{-1}.
\end{align}
So, we only need to take care of the  remaining term $I_2$.
For this, consider the event that the time-continuous Euler-Maruyama scheme in one step does not move farther than $d(\appx_{\et}, \Theta)/2$ away from $\appx_{\et}$, that is 
\begin{align} \label{eq:A}
A(t)=   \Big \{ \|\appx_t - \appx_{\et}\| \leq d(\appx_{\et}, \Theta)/2 \Big \}  ,  \quad t\ge 0.
\end{align}
We use the following partition of 1:
\begin{align*}
 \1_{ \{\appx_{\et}  \in \hypsurf^{\varepsilon_1} \setminus  \hypsurf^{\varepsilon_2} \}}=
  \1_{ A(t)\cap\{\appx_{\et}  \in \hypsurf^{\varepsilon_1} \setminus  \hypsurf^{\varepsilon_2} \}}
  + \1_{ A(t)^c \cap\{\appx_{\et}  \in \hypsurf^{\varepsilon_1} \setminus  \hypsurf^{\varepsilon_2} \}}.
\end{align*}

\begin{enumerate}[(i)]
\item
The distance function $d(\cdot,\hypsurf)$  to the hypersurface  $\hypsurf$ is Lipschitz continuous with Lipschitz constant 1, see \cite[Equation (14.91)]{trud}, so we have
\[
d(x, \Theta) - d(y, \Theta) 
\leq  |  d(x, \Theta) - d(y, \Theta)  |
\leq  \|  x-y   \|, \quad x,y \in \hypsurf^{\varepsilon_0}.
\]
Hence, we observe that 
\[
\begin{split}
& d(x, \Theta) \leq 
 \|  x  - y  \|  + d(y, \Theta) , \quad x,y \in \hypsurf^{\varepsilon_0},
\end{split}
\]
which implies
\[
A(t) \cap \{\appx_{\et}\in \Theta^{\varepsilon_1} \setminus \Theta^{\varepsilon_2}   \}  \subseteq  \left \{ \frac{1}{2} d(\appx_{\et}, \Theta)  \leq d(\appx_{t}, \Theta) \leq \frac{3}{2} d(\appx_{\et}, \Theta)  \right \} \cap \{\appx_{\et}\in \Theta^{\varepsilon_1} \setminus \Theta^{\varepsilon_2}   \} .
\]
It follows that
\begin{align*}
 \frac{1}{h(\appx_{\et}, \delta)} \1_{A(t)\cap \{\appx_{\et}\in \Theta^{\varepsilon_1} \setminus \Theta^{\varepsilon_2}   \} }
& = \frac{\|\sigma\|^2_{\infty,\hypsurf^{\varepsilon_0}} (\log(\delta))^2}{d(\appx_\et,\hypsurf)^2}  \1_{A(t)\cap \{\appx_{\et}\in \Theta^{\varepsilon_1} \setminus \Theta^{\varepsilon_2}   \} } 
 \\ &  \le \frac{9}{4}\frac{\|\sigma\|^2_{\infty,\hypsurf^{\varepsilon_0}} (\log(\delta))^2}{d(\appx_t,\hypsurf)^2}  \1_{A(t)\cap \{\appx_{\et}\in \Theta^{\varepsilon_1} \setminus \Theta^{\varepsilon_2}   \}}. 
 \end{align*}
Moreover,  
\begin{align*}
  \{ \appx_{\et}\in \Theta^{\varepsilon_1} \setminus \Theta^{\varepsilon_2}   \} \cap  A(t) \subseteq \left \{ \appx_{t} \in \hypsurf^{\frac{3}{2}\varepsilon_1} \setminus  \hypsurf^{\frac{1}{2} \varepsilon_2} \right \}
\end{align*}
and so we obtain
\begin{align} \label{eq:cost-help2}
 &\frac{1}{h(\appx_{\et}, \delta)} \1_{A(t)\cap \{\appx_{\et}\in \Theta^{\varepsilon_1} \setminus \Theta^{\varepsilon_2}   \} }
  \le \frac{9}{4} \frac{\|\sigma\|^2_{\infty,\hypsurf^{\varepsilon_0}} (\log(\delta))^2}{d(\appx_t,\hypsurf)^2}  \1_{ \left \{  \appx_{t} \in \hypsurf^{\frac{3}{2}\varepsilon_1} \setminus  \hypsurf^{\frac{1}{2} \varepsilon_2} \right  \}}.
 \end{align}
 
\item Since the minimal step size is $\delta^2$, we have
\begin{align}\label{eq:cost-help3}
 \frac{1}{h(\appx_{\et}, \delta)}  \1_{A(t)^c\cap \{\appx_{\et} \in \Theta^{\varepsilon_1} \setminus \Theta^{\varepsilon_2}   \} }&\leq  \frac{1}{\delta^{2}} \1_{A(t)^c\cap \{\appx_{\et}\in \Theta^{\varepsilon_1} \setminus \Theta^{\varepsilon_2}   \} } \\
 &= \frac{1}{\delta^{2}} \1_{ \{\|\appx_t- \appx_{\et}\| > d(\appx_{\et}, \Theta)/2; \appx_{\et}(\omega) \in \Theta^{\varepsilon_1} \setminus \Theta^{\varepsilon_2}   \} } .
 \end{align}
 
\end{enumerate}

Combining \eqref{eq:Is} with \eqref{eq:cost-help2} and \eqref{eq:cost-help3} we obtain
\begin{equation} \label{eq:est-I2}
\begin{aligned}
 I_2 &  \leq \frac{9}{4} \|\sigma\|^2_{\infty,\hypsurf^{\varepsilon_0}} (\log(\delta))^2 \int_{0}^{T}  \E \left[ \frac{1}{d(\appx_t,\hypsurf)^2}\1_{  \left \{ \, \appx_{t} \in \hypsurf^{\frac{3}{2}\varepsilon_1} \setminus  \hypsurf^{\frac{1}{2} \varepsilon_2} \right  \} } \right]dt\\
&\quad+ \delta^{-2}  \int_{0}^{T} \E\left[  \1_{ \{ \appx_{\et} \in \Theta^{\varepsilon_1} \setminus \Theta^{\varepsilon_2};   \|\appx_{t} - \appx_{\et}\| > d(\appx_{\et},\Theta)/2\} }\right] dt \\
& = \frac{9}{4} \|\sigma\|^2_{\infty,\hypsurf^{\varepsilon_0}} (\log(\delta))^2 \int_{0}^{T}  \E \left[ \frac{1}{ \max \{  \varepsilon_2/2, d(\appx_t,\hypsurf) \}^2}\1_{  \left \{ \, \appx_{t} \in \hypsurf^{\frac{3}{2}\varepsilon_1}  \right  \} } \right]dt\\
&\quad+ \delta^{-2}  \int_{0}^{T} \E\left[  \1_{ \{ \appx_{\et} \in \Theta^{\varepsilon_1} \setminus \Theta^{\varepsilon_2};   \|\appx_{t} - \appx_{\et}\| > d(\appx_{\et},\Theta)/2\} }\right] dt \\
 &=: I_{21} + I_{22}.
\end{aligned}
\end{equation} 
By  Lemma \ref{lem:regimes_extra} with $\alpha=1/2$  we obtain that there exists a constant $c_2>0$ such that
\begin{align} \label{est31}
I_{22} \leq c_{2} \cdot(1+ \log(1/\delta)) \delta^{-1}.
 \end{align}
 Lemma \ref{it:inner_cont} yields that there exist constants $c_3,c_4>0$ such that

\begin{equation}\label{est32}
 I_{21} \leq c_3 \left( 2 \varepsilon_2^{-1 }+ \int^{\frac{3}{2}\varepsilon_1}_{\frac{1}{2}\varepsilon_2} \frac{\|\sigma\|^2_{\infty,\hypsurf^{\varepsilon_0}} (\log(\delta))^2}{x^2} dx \right)
 \le c_4 \cdot \log(1/\delta) \delta^{-1}.
  \end{equation}
 Combining \eqref{eq:est-I2} with \eqref{est31} and \eqref{est32} ensures the existence of a constant $c_5>0$ such that
 \begin{align} \label{est3}
I_{2} \leq c_5 \cdot  (1+\log(1/\delta)) \delta^{-1}.
 \end{align}
 
Now, \eqref{eq:est-N} together with the estimates \eqref{est1}, \eqref{est2}, and \eqref{est3} yields the assertion.
\end{proof}

\section{Examples}
\label{sec:examples}

In this section we present some numerical examples to complement our asymptotic convergence analysis with a study of the non-asymptotic regime.
For all examples we choose for simplicity $T=1$. We use
$$ \operatorname{cost}(\delta)= \frac{1}{M} \sum_{i=1}^{M}  N(h,\delta)^{(i)},$$
where $M\in\N$ is the sample size,
 to estimate the computational cost and
$$    \operatorname{msq}(\delta) =  \frac{1}{M} \sum_{i=1}^{M}  \left\| \big(X_1^{h(\cdot,\delta)} - X_1^{h(\cdot, 2\delta)} \big)^{(i)} \right\|^2$$
to estimate the convergence rate. The latter is justified (for dyadic $\delta$) by the fact that if there exist $\beta\in\R$, $\gamma\in(0,\infty)$ such that
\begin{align*}
 \limsup_{n \rightarrow \infty} \big( n^{\beta}(2^{n})^{\gamma} \big) \cdot  \E \left \|X_1-X_1^{h(\cdot,2^{-n})} \right \|^2 < \infty,
  \end{align*}
  then also
 \begin{align*}
\limsup_{n \rightarrow \infty} \big( n^{\beta} (2^{n})^{\gamma} \big) \cdot \E \left \|X_1^{h(\cdot,2^{-n-1})}-X_1^{h(\cdot,2^{-n})} \right \|^2 < \infty,
\end{align*}
and vice versa. Above we use the standard convention that $Y^{(i)}$ denotes an iid copy of a random variable $Y$.

For both quantities we choose $\delta=2^{-2},2^{-3}, \ldots, 2^{-10}$, $M=5\cdot 10^4$ and perform a regression 
using the ansatz $$ f(\delta)= c_1 \cdot \log(1/\delta)^{c_2} \cdot \delta^{c_3}$$ to determine $c_1>0$ and $c_2,c_3 \in \mathbb{R}$.\\

Our first test equation is a scalar equation with
$$ \mu(x)= -2 \cdot \1_{(-\infty,0)}(x) + x^2 \cdot \1_{[0,1)}(x)+ \left( \frac{2}{x}-\frac{3}{x^2} \right) \cdot \1_{[1,\infty)}(x),  \qquad \sigma(x)=0.5 \cdot \left(1+ \frac{1}{1+x^2} \right),$$
and initial value $x=1.5$.
For the cost of the Euler-Maruyama scheme we obtain
$$f_{\operatorname{cost}}(\delta)= 1.2014 \cdot  \log(1/\delta)^{0.8936} \cdot \delta^{-1.1218} ,$$
with a residuum of ${\tt res}=6.0412\cdot 10^2$,
while for the mean square error we have
$$ f_{\operatorname{msq}}(\delta) =  0.5940\cdot\log(1/\delta)^{-2.0209} \cdot \delta^{1.1037} ,$$
with a residuum of ${\tt res}=1.0674 \cdot 10^{-2}$.
This is in good accordance with the predicted asymptotic behaviour from Theorem \ref{thm:conv}, respectively  Theorem \ref{thm:cost}. \\

The second  test equation is again a scalar equation with
$$ \mu(x)= -1 \cdot \1_{(-\infty,-1)}(x) +1 \cdot \1_{[-1,2)}(x) -2x \cdot \1_{[2,\infty)}(x),  \qquad \sigma(x)=1,$$
and $x=0$, i.e.~an equation with additive noise.
Here we have
$$f_{\operatorname{cost}}(\delta)=    0.9148 \cdot  \log(1/\delta)^{0.5163} \cdot \delta^{-1.1380} ,$$
with a residuum of ${\tt res}=2.0217 \cdot 10^2$ and
$$ f_{\operatorname{msq}}(\delta) =  
  21.2638 \cdot\log(1/\delta)^{-1.8354} \cdot \delta^{1.5232} ,$$
with a residuum of ${\tt res}=2.2471\cdot 10^{-2}$.
The increase in the observed empirical convergence order is not surprising. Convergence order $3/4$ for the equidistant Euler-Maruyama scheme with additive noise has already been indicated by the simulation results in \cite{GLN2017}.
However, the latter work also indicates that discontinuous coefficients may lead to an unstable behaviour of $ {\operatorname{msq}}(\delta)$. 
To determine sharp bounds for the convergence order of the Euler-Maruyama scheme for  SDEs with additive noise and discontinuous coefficients will be part of our future research.\\

 The final test equation is two-dimensional with degenerate noise, i.e.~we have
\begin{align*}
 \mu(x_1,x_2)=\begin{cases}
(1,1)^\top, & x_1^2+x_2^2 \geq 1,\\
(-x_1,x_2)^\top, & x_1^2+x_2^2 < 1 ,
          \end{cases}
\qquad 
 \sigma(x_1,x_2)= \frac{1}{2}\begin{pmatrix}x_1 & 0 \\ x_2 & 0 \end{pmatrix},
         \end{align*}
and initial value $x=(0.5,0.5)^\top$.
We obtain
$$f_{\operatorname{cost}}(\delta)=    1.7280 \cdot  \log(1/\delta)^{0.7362} \cdot \delta^{-1.0248}, $$
with a residuum of ${\tt res}=2.3685 \cdot 10^1$ and
$$ f_{\operatorname{msq}}(\delta) =  
  11.9163 \cdot\log(1/\delta)^{-2.2178} \cdot \delta^{1.0389} ,$$
with a residuum of ${\tt res}=1.6205 \cdot 10^{-1}$.
This is again in good accordance with our analysis.


\section*{Acknowledgements}
The authors are grateful to Mike Giles for numerous discussions. The timestepping strategy of this paper has been inspired by work of M.~Giles and K.~Ramanan on adaptive Multilevel Monte Carlo timestepping strategies for reflected diffusions.
Furthermore, the authors are also grateful to Martin Schmidt for a discussion about  geometric aspects of this work
and to Larisa Yaroslavtseva for pointing out to us an inaccuracy in the proof of a lemma.

M.~Sz\"olgyenyi is supported by the AXA Research Fund grant "Numerical Methods for Stochastic Differential Equations with Irregular Coefficients with Applications in Risk Theory and Mathematical Finance".
A part of this article was written while M.~Sz\"olgyenyi was affiliated with the Seminar for Applied Mathematics and the RiskLab Switzerland, ETH Zurich, R\"amistrasse 101, 8092 Zurich, Switzerland and with the Institute of Statistics and Mathematics, Vienna University of Economics and Business, Welthandelsplatz 1, 1020 Vienna, Austria, and supported by the Vienna Science and Technology Fund (WWTF): Project MA14-031.

Moreover, part of this work was carried out while A.~Neuenkirch and M.~Sz\"olgyenyi were guests of the Erwin Schr\"odinger International Institute (ESI) for Mathematics and Physics in Vienna, whose support and hospitality is gratefully acknowledged.



\vspace{2em}
\centerline{\underline{\hspace*{16cm}}}

 \noindent Andreas Neuenkirch  \\
Institut f\"ur Mathematik, Universit\"at Mannheim,
B6, 26, 
68131 Mannheim, Germany\\
neuenkirch@math.uni-mannheim.de\\

\noindent Michaela Sz\"olgyenyi \\
Department of Statistics, University of Klagenfurt, Universit\"atsstra\ss{}e 65--67,
9020 Klagenfurt, Austria\\
michaela.szoelgyenyi@aau.at\\

 \noindent Lukasz Szpruch \\
 School of Mathematics, University of Edinburgh,
Peter Guthrie Tait Road,
Edinburgh EH9 3FD,  Great Britain\\
 l.szpruch@ed.ac.uk \\


\end{document}